\documentclass[11pt,a4paper]{amsart}
\usepackage{amssymb,xspace}
\usepackage{amstext}
\usepackage{tikz}
\usetikzlibrary{arrows,decorations.pathmorphing,backgrounds,fit,positioning,shapes.symbols,chains}
\theoremstyle{plain}
\usepackage{amsbsy,amssymb,amsfonts,latexsym}

\usepackage{enumerate}
\usepackage{graphics}

\date{\today}
\title{How behave the typical $L^q$-dimensions of measures?}
\author{Fr\'ed\'eric Bayart}
\address{
Clermont Universit\'e, Universit\'e Blaise Pascal, Laboratoire de Math\'ematiques, BP 10448, F-63000 CLERMONT-FERRAND -
CNRS, UMR 6620, Laboratoire de Math\'ematiques, F-63177 AUBIERE
}
\email{Frederic.Bayart@math.univ-bpclermont.fr}



\newcommand{\veps}{\varepsilon}

\def\ZZ{\mathbb Z}

\def\card{\textrm{card}}
\def\dsupmuq{\overline{D}_\mu(q)}
\def\dinfmuq{\underline{D}_\mu(q)}
\def\dinfmu{\underline{D}_\mu}
\def\dsupmu{\overline{D}_\mu}
\def\dsupnuq{\overline{D}_\nu(q)}
\def\dinfnuq{\underline{D}_\nu(q)}
\def\dinfnu{\underline{D}_\nu}
\def\dsupnu{\overline{D}_\nu}
\def\pk{\mathcal P(K)}
\def\fk{\mathcal F(K)}
\def\dboxinf{\underline{\dim}_B}
\def\dboxinfloc{\underline{\dim}_{B,{\rm loc}}}
\def\dboxsup{\overline{\dim}_B}
\def\dboxsuploc{\overline{\dim}_{B,{\rm loc}}}
\def\dboxsupconv{\overline{\dim}_{B,{\rm conv}}}
\def\dboxsupconvmax{\overline{\dim}_{B,{\rm conv,max}}}
\def\sconv{s_{\rm conv}}
\def\sconvmax{s_{{\rm conv}}^{\rm \max}}
\def\dboxsuplocunif{\overline{\dim}_{B,{\rm loc},{\rm unif}}}

\def\dimh{\dim_{\mathcal H}}
\def\dimp{\dim_{\mathcal P}}

\DeclareMathOperator{\supp}{supp}
\DeclareMathOperator{\bsi}{bsi}

\newtheorem{theorem}{Theorem}[section]

\newtheorem{lemma}[theorem]{Lemma}

\newtheorem{proposition}[theorem]{Proposition}

\newtheorem{corollary}[theorem]{Corollary}

{\theoremstyle{definition}}
{\theoremstyle{definition}}

{\theoremstyle{definition}\newtheorem{example}[theorem]{Example}}

{\theoremstyle{definition}\newtheorem{definition}[theorem]{Definition}}

{\theoremstyle{definition}}

{\theoremstyle{definition}\newtheorem{remark}[theorem]{Remark}}

\newtheorem{question}[theorem]{Question}

\newtheorem*{OLSEN}{Theorem A (Olsen)}

\begin{document}

\begin{abstract}
We compute, for a compact set $K\subset\mathbb R^d$, the value of the upper and of the lower $L^q$-dimension of a typical
probability measure with support contained in $K$, for any $q\in\mathbb R$. Different definitions of the ``dimension'' of $K$ are involved to
compute these values, following $q\in\mathbb R$.
\end{abstract}

\maketitle

\section{Introduction}
Let $K$ be a compact subset of $\mathbb R^d$, and let $\mathcal P(K)$ be the set of Borel probability measures on $K$;
we endow $\mathcal P(K)$ with the weak topology.
In this paper, we are interested in properties shared by typical measures of $\pk$. By a property true for a typical measure
of $\pk$, we mean a property which is satisfied by a dense $G_\delta$ set of elements of $\pk$.

Specifically we deal with the upper and lower $L^q$-dimensions of measures. Let $\mu\in\pk$, $r>0$ and $q\in\mathbb R\backslash\{1\}$. We write
$$I_\mu(r,q)=\int_K \mu\big( B(x,r)\big)^{q-1}d\mu(x).$$
The lower and upper $L^q$-dimensions are now defined, for $q\neq 1$, by
\begin{eqnarray*}
 \dinfmuq&=&\liminf_{r\to 0}\frac{1}{q-1}\times\frac{\log I_\mu(r,q)}{\log r}\\
\dsupmuq&=&\limsup_{r\to 0}\frac{1}{q-1}\times\frac{\log I_\mu(r,q)}{\log r}.
\end{eqnarray*}
When $q=1$, the definitions involve a logarithmic factor:
\begin{eqnarray*}
I_\mu(r,1)&=&\int_K \log \mu\big(B(x,r)\big)d\mu(x)\\
 \dinfmu(1)&=&\liminf_{r\to 0}\frac{\int_{K}\log\mu\big(B(x,r)\big)d\mu(x)}{\log r}\\
\dsupmu(1)&=&\limsup_{r\to 0}\frac{\int_{K}\log\mu\big(B(x,r)\big)d\mu(x)}{\log r}.
\end{eqnarray*}

These dimensions were introduced by Hentschel and Procaccia in \cite{HP83} in order to generalize the information dimension of a measure.
They are important for their relationship with the multifractal formalism. This formalism, which was conjectured
in the 1980s in the physic literature, asserts that for ``good'' measures, $\dinfmuq=\dsupmuq$ for any $q\in\mathbb R$ and that
the Hausdorff multifractal spectrum of $\mu$ coincides with the Legendre transform of $\tau_\mu:\mathbb R\to\mathbb R$ defined by 
$\tau_\mu(q)=(1-q)\dinfmuq=(1-q)\dsupmuq$. 
This multifractal formalism has been verified for various classes of measures on $\mathbb R^d$, see \cite{Fal97} and the references therein.

\medskip

In a series of papers (\cite{OL05}, \cite{Ol07} and \cite{OL08}), L. Olsen made a first approach to estimate the typical
value of $\dinfmuq$ and of $\dsupmuq$. To state his results, we need to introduce some classical terminology which can be found e.g. in \cite{Fal03}.
For a subset $E\subset\mathbb R^d$, we denote the lower box dimension of $E$ and the upper box dimension of $E$
by $\dboxinf(E)$ and $\dboxsup(E)$, respectively. The Hausdorff and the packing dimension of $E$ are denoted respectively by 
$\dimh(E)$ and $\dimp(E)$. Also, for a subset $K$ of $\mathbb R^d$ and $x\in K$, we
define the lower local box dimension of $K$ at $x$ and the upper local box dimension of $K$ at $x$ by
\begin{eqnarray*}
 \dboxinfloc(x,K)&=&\lim_{r\to 0}\dboxinf\big(K\cap B(x,r)\big)\\
\dboxsuploc(x,K)&=&\lim_{r\to 0}\dboxsup\big(K\cap B(x,r)\big).
\end{eqnarray*}
\begin{OLSEN}
 Let $K$ be a compact subset of $\mathbb R^d$. Write
\begin{eqnarray*}
 s_-&=&\inf_{x\in K}\dboxinfloc(x,K)=\inf_{x\in K}\inf_{r>0}\dboxinf\big(B(x,r)\big)\\
s^+&=&\inf_{x\in K}\dboxsuploc(x,K)=\inf_{x\in K}\inf_{r>0}\dboxsup\big(B(x,r)\big)\\
s&=&\dboxsup(K).
\end{eqnarray*}
Then the following holds:
$$\begin{array}{c|c|c|}
   &\textrm{All measures $\mu\in\pk$ satisfy}&\textrm{A typical measure $\mu\in\pk$ satisfies} \\  \hline
&&\\
q> 1& 0\leq \dinfmuq\leq\dsupmuq\leq s& \dinfmuq=0\\
&&s_-\leq \dsupmuq\leq s^+ \\
&&\\ \hline
&&\\
q\in [0,1]&0\leq \dinfmuq\leq\dsupmuq\leq s&s_-\leq\dsupmuq\leq s.\\ 
&&\\
  \end{array}$$
\end{OLSEN}
\begin{remark} In this paper, we adopt the normalization of \cite{Ol07} instead of that of \cite{OL05}, because
 it avoids unpleasant factors $(1-q)$ all over the proofs, and also because it seems clearer if a dimension is always nonnegative.
We nevertheless mention that this inverts the definition of $\dinfmuq$ and $\dsupmuq$ of \cite{OL05} when $q>1$.
\end{remark}

\begin{remark} In \cite{OL08}, the inequality $s_-\leq\dsupmuq\leq s^+$ is announced for a typical measure when $q\in[0,1)$.
However, there is a mistake in the proof, precisely in Lemma 2.3.1. A correct statement of this lemma should be:
\begin{quote}
 Let $E\subset K$ and let $\mu\in\mathcal P(E)$. Then 
$$\dsupmuq\leq \dboxsup(E)$$
for all $q\in [0,1)$.
\end{quote}
This correct statement forces to replace $s^+$ by $s$ in Theorem A, which is less good. See also
the forthcoming Theorem \ref{THMMAINLQSPECTRUM}, which shows that typically $\dsupmuq\geq s^+$ for $q\in[0,1)$.
\end{remark}
\begin{remark}
 It should be observed that the precise statement of Theorem A is:
 \begin{quote} a typical measure satisfies for every $q> 1\dots$
 \end{quote} 
 That is formally
stronger than:
\begin{quote}
 for every $q> 1$, a typical measure satisfies $\dots$
 \end{quote}
\end{remark}

\medskip

The work of Olsen leaves open several questions: 
\begin{itemize}
 \item What happens in the remaining cases? Olsen conjectured that, for any $q<1$, $\dinfmuq=0$ and that, for any $q<0$, $\dsupmuq=+\infty$.
 \item Can we say more for the upper $L^q$-dimension? Precisely, does there exist for any $q\geq 0$ a real number $f(q)$ such that 
a typical measure
$\mu\in\pk$ satisfies $\dsupmuq=f(q)$? In this case, Olsen conjectured that we cannot do better than $s_-\leq \dsupmuq\leq s^+$ for a typical
$\mu\in\pk$.
\end{itemize}

\smallskip

Our aim, in this paper, is to answer these questions. To this intention, we need to introduce two new ways to measure the size of a compact set.
The first one measures how behaves locally the upper box dimension, uniformly in $K$. 
For a set $E\subset\mathbb R^d$ and $r>0$, we denote by $\mathbf P_r(E)$ the largest number of pairwise disjoint balls of radius $r$ with centers in $E$
and by $\mathbf  N_r(E)$ the smallest number of balls of radius $r$ which are needed to cover $E$.
\begin{definition}
 Let $K$ be a compact subset of $\mathbb R^d$. The \emph{local uniform upper box dimension} of $K$ is the real number $\dboxsuplocunif(K)$ defined by
\begin{eqnarray*}
 \dboxsuplocunif(K)&=&\inf_{N\geq 1}\inf_{x_1,\dots,x_N\in K}\inf_{r_0> 0}\limsup_{r\to 0}\inf_{i=1,\dots,N}\frac{\log \mathbf P_r\big(K\cap B(x_i,r_0)\big)}{-\log r}\\
&=&\inf_{N\geq 1}\inf_{x_1,\dots,x_N\in K}\inf_{r_0> 0}\limsup_{r\to 0}\inf_{i=1,\dots,N}\frac{\log \mathbf N_r\big(K\cap B(x_i,r_0)\big)}{-\log r}
\end{eqnarray*}
\end{definition}
(the equivalence between these two definitions is a standard property of box-like dimensions.) It is easy to check that
$$\inf_{x\in K}\dboxinfloc(x,K)\leq \dboxsuplocunif(K)\leq \inf_{x\in K}\dboxsuploc(x,K).$$
We shall see later that these inequalities can be strict.

\smallskip

We will also need another variant of the lower box dimension, which measures both the size and the connectivity of a set. For $K$ a compact subset of
$\mathbb R^d$, let $\mathcal C_n(K)$ be the collection of the half-closed dyadic cubes of size $2^{-n}$ intersecting $K$, namely
$$\mathcal C_n(K)=\left\{\prod_{j=1}^n \left[\frac{k_j}{2^n};\frac{k_j+1}{2^n}\right);\ k_j\in\ZZ\textrm{ and }\prod_{j=1}^n \left[\frac{k_j}{2^n};\frac{k_j+1}{2^n}\right)\cap K\neq \varnothing\right\}.$$
 Let $K_n=\bigcup_{C\in\mathcal C_n(K)}C$
and let $C_n(K)$ be the number of connected components of $K_n$. We may observe that the sequence $(C_n(K))$ is nondecreasing.

\begin{definition}
 Let $K$ be a compact subset of $\mathbb R^d$. 
 The \emph{box separation index} of $K$ is the real number $\bsi(K)$ defined by 
$$\bsi(K)=\liminf_{n\to+\infty}\frac{\log(C_n(K))}{n\log 2}.$$
\end{definition}
It is clear that $\bsi(K)\leq\dboxinf(K)$. Heuristically speaking, $\bsi(K)$ is large if, for each $n\geq 1$, you need many cubes of size $2^{-n}$ to cover it and if these cubes
are far away from each other. 

\smallskip

It should be pointed out that neither the local uniform upper box dimension nor the box separation index can be considered as a dimension.
For instance, $E\subset F$ does not imply $\dboxsuplocunif(E)\leq\dboxsuplocunif(F)$ or $\bsi(E)\leq \bsi(F)$. 
As an example, if you set $E=[1,2]$ and $F=\{0\}\cup [1,2]$, then $\dboxsuplocunif(E)=1$ whereas  $\dboxsuplocunif(F)=0$.
Regarding the box separation index, $\bsi([0,1])=0$ (the $2^{-n}$-mesh cubes intersecting $[0,1]$ are connected), whereas $\bsi(K)=1/2$
when $K=\{0\}\cup\{1/n;\ n\geq 1\}$. This last fact follows easily from the standard proof of $\dboxinf(K)=1/2$.

\bigskip

Our first main result now reads:
\begin{theorem}\label{THMMAINLQSPECTRUM}
 Let $K$ be an infinite compact subset of $\mathbb R^d$. Write
\begin{eqnarray*}
 s_{\rm sep}&=&\bsi(K)\\
s_u&=&\dboxsuplocunif(K)\\
s&=&\dboxsup(K)\\
s_\mathcal P&=&\dimp(K).
\end{eqnarray*}
Then the following holds:
$$\begin{array}{c|c|c|}
   &\textrm{All measures $\mu\in\pk$ satisfy}&\textrm{A typical measure $\mu\in\pk$ satisfies} \\  \hline
&&\\
q> 1& 0\leq \dinfmuq\leq\dsupmuq\leq s_\mathcal P &\dsupmuq=s_u\\
&& \dinfmuq=0 \\
&&\\ \hline
&&\\
q\in (0,1)&0\leq \dinfmuq\leq\dsupmuq\leq s&\dsupmuq=s\\
&&\dinfmuq=0\\ 
&&\\ \hline
&&\\
q=0&0\leq \dinfmuq\leq\dsupmuq\leq s&\dsupmuq=s\\
&&\dinfmuq=s_{\rm sep}\\ 
&&\\ \hline
&&\\
q<0&0\leq \dinfmuq\leq\dsupmuq&\dsupmuq=+\infty\\
&&\dinfmuq=s_{\rm sep}.\\ 
&&\\ 
\end{array}
$$
\end{theorem}

The typical values of $\dsupmu(q)$, for $q>1$, and of $\dinfmuq$, for $q\leq 0$, are very interesting. Indeed, they did not take the worst possible values, what is rather atypical in analysis! 
From a technical point of view, these results force us to prove two inequalities instead of just one, since we cannot use the results of the second column.
They also point out that the typical $L^q$-dimensions of measures supported by $K$ depend heavily both on the size and on the local structure of $K$,
since they involve the local uniform upper box dimension and the box separation index of $K$.

\smallskip

As one can guess, the proof of Theorem \ref{THMMAINLQSPECTRUM} requires delicate constructions of measures as well as a careful
examination of the local structure of a compact set. More surprizingly, it also involves elements of graph theory. It should be noted that
the unusual notions of dimension used in Theorem \ref{THMMAINLQSPECTRUM} can be easily computed for natural compact sets arising in multifractal analysis:
Cantor sets, self-similar compact sets, finite disjoint unions of these sets,...

\bigskip

We turn now to the $L^1$-case. This is maybe the most important case. In the literature, the $L^1$-dimensions are also known as 
the (upper and lower) information dimension or as the (upper and lower) entropy. 
From Theorem \ref{THMMAINLQSPECTRUM}, we can get immediately the typical value of $\dinfmu(1)$. However, that of $\dsupmu(1)$ is unpredictable: it is not clear
if it should be $s_u$, like for $q>1$, or $s$, like for $q<1$.  It turns out that neither $s_u$ nor $s$ is convenient. In fact, the situation breaks down dramatically
for the upper $L^1$-dimension: in general, there is no typical value for $\dsupmu(1)$! We cannot say more that $\dsupmu(1)$ belongs typically to some interval, and we shall give soon
the optimal interval. Nevertheless, we need to introduce yet another couple of definitions.

\begin{definition}
Let $K$ be a compact subset of $\mathbb R^d$.  The \emph{convex upper box dimension} of $K$ is the real number $\dboxsupconv(K)$ defined by
$$ \dboxsupconv(K)=\inf_{N\geq 1}\inf_{x_1,\dots,x_N\in K}\inf_{\rho> 0}\inf_{p_i>0,\sum_i p_i=1}\limsup_{r\to 0}\frac{\sum_{i}p_i\log \mathbf P_r\big(K\cap B(x_i,\rho)\big)}{-\log r}.$$
The \emph{maximal convex upper box dimension} of $K$ is the real number $\dboxsupconvmax(K)$ defined by
$$ \dboxsupconvmax(K)=\sup_{y\in K,\rho>0}\dboxsupconv\big(K\cap B(y,\rho)\big).$$
\end{definition}

It is clear from the definition that
$$\dboxsuplocunif(K)\leq \dboxsupconv(K)\leq \dboxsupconvmax(K).$$
We shall see later that these inequalities can be strict.

\smallskip

Our main theorem on $L^1$ now reads:
\begin{theorem}\label{THMMAINL1}
Let $K$ be an infinite compact subset of $\mathbb R^d$. Then a typical measure $\mu\in\pk$ satisfies
$$\dinfmu(1)=0\textrm{ and }\dsupmu(1)\in[\dboxsupconv(K),\dboxsupconvmax(K)].$$
Moreover, if $[a,b]$ is any interval such that a typical measure $\mu\in\pk$ satisfies $\dsupmu(1)\in[a,b]$, then 
$$a\leq \dboxsupconv(K)\textrm{ and }b\geq \dboxsupconvmax(K).$$
\end{theorem}
As before, even though the definition of the (maximal) convex upper box dimension is not very appealing, it can be easily computed
for many compact sets, like Cantor sets, self-similar compact sets, finite disjoint unions of these sets...
Here is an example to see how Theorem \ref{THMMAINL1} reads on 
a very easy compact set.
It points out that we cannot expect to get a typical value for the upper $L^1$-dimension. This happens only for this value of $q$.
\begin{example}
 Let $K=\{0\}\cup[1,2]$. Then a typical measure $\mu\in\pk$ satisfies 
 $\dsupmu(1)\in[0,1]$ and this interval is the best possible.
\end{example}

\bigskip

The paper is organized as follows. In Section 2, we introduce the tools which are needed throughout the paper.
In Section 3,4,5, we prove Theorem \ref{THMMAINLQSPECTRUM}, whereas Section 6 is devoted to the study of the 
upper $L^1$-dimension. We conclude in Section 7 by remarks and open questions.

\section{Preliminaries} \label{SECPRELIMINARIES}
\subsection{The topology on $\mathcal P(K)$}

Throughout this paper, $\mathcal P(K)$ will be endowed with the weak topology. It is well known (see for instance \cite{Par67})
that this topology is completely metrizable by the Fortet-Mourier distance defined as follows. Let $\textrm{Lip}(K)$ denote the family 
of Lipschitz functions $f:K\to\mathbb R$, with $|f|\leq 1$ and $\textrm{Lip}(f)\leq 1$, where $\textrm{Lip}(f)$ denotes
the Lipschitz constant of $f$. The metric $L$ is defined by
$$L(\mu,\nu)=\sup_{f\in\textrm{Lip}(K)}\left|\int fd\mu-\int fd\nu\right|$$
for any $\mu,\nu\in\mathcal P(K)$. We endow $\mathcal P(K)$ with the metric $L$. In particular, for $\mu\in\mathcal P(K)$
and $\delta>0$, $B_{L}(\mu,\delta)=\{\nu\in\mathcal P(K);\ L(\mu,\nu)<\delta\}$ will stand for the ball with center at $\mu$ and radius
equal to $\delta$.

We shall use several times the following lemma.
\begin{lemma}\label{LEMTOPO1}
 For any $\alpha\in(0,1)$, for any $\beta>0$, there exists $\delta>0$ such that, for any $E$ a Borel subset of $K$, for any
$\mu,\nu\in\mathcal P(K)$, 
$$L(\mu,\nu)<\delta\implies \mu(E)\leq\nu\big(E(\alpha)\big)+\beta,$$
where $E(\alpha)=\{x\in K;\ \textrm{dist}(x,E)<\alpha\}$.
\end{lemma}
\begin{proof}
 We set
$$f(t)=\left\{
\begin{array}{ll}
 \alpha&\textrm{provided }t\in \overline E\\
\alpha-\textrm{dist}(x,E)&\textrm{provided }0<\textrm{dist}(x,E)\leq\alpha\\
0&\textrm{otherwise}.
\end{array}\right.$$
Then $f$ is Lipschitz, with $|f|\leq 1$ and $\textrm{Lip}(f)\leq 1$. Thus,
\begin{eqnarray*}
 \mu(E)&\leq&\frac 1\alpha\int fd\mu\\
&\leq&\frac 1\alpha \left[\int fd\nu+\delta\right]\\
&\leq&\nu \big(E(\alpha)\big)+\frac\delta\alpha.
\end{eqnarray*}
Hence, it suffices to take $\delta=\alpha\beta$.
\end{proof}

Our first application of Lemma \ref{LEMTOPO1} is that a small perturbation of a finite measure $\mu$ does not change dramatically
the value of $I_\mu(r,q)$, provided we allow to change slightly the radius. Here is the statement that we can get for $q> 1$. $\fk$ denotes the set of probability measures with finite support in $K$.
\begin{corollary}\label{CORTOPO4}
Let $q> 1$. There exists $C_q>0$ such that, for any $\mu\in\mathcal F(K)$, for any $r>0$, there exists $\delta>0$ such that, for any  $\nu\in\pk$ with $L(\mu,\nu)<\delta$, 
$$I_\nu(r,q)\leq C_q I_\mu(2r,q)\textrm{ and }I_\nu(2r,q)\geq C_q^{-1}I_\mu(r,q).$$
\end{corollary}
\begin{proof}
We begin by fixing a pair $(t,\eta)$ with $t\in(r,2r)$ and $\eta>0$ such that
$$\left\{
\begin{array}{l}
 [t-\eta,t+\eta]\cap \big\{\|x-y\|;\ x,y\in\supp(\mu)\big\}=\varnothing;\\
\forall x,y\in\supp(\mu),\ x\neq y\implies \|x-y\|>2\eta.
\end{array}
\right.$$
Observe that this choice of $t$ and $\eta$ guarantees that, 
for any $x\in\supp(\mu)$ and any $z\in B(x,\eta)$, then
$$\supp(\mu)\cap B(x,t)=\supp(\mu)\cap B(z,t)$$
so that $\mu\big(B(x,t)\big)=\mu\big(B(z,t)\big).$ Now, suppose that $\delta>0$ has been chosen so small that any $\nu\in\pk$ with $L(\mu,\nu)<\delta$
satisfies
\begin{itemize}
\item for any $x\in\supp(\mu)$, 
$$\frac12 \mu(\{x\})=\frac 12 \mu\big(B(x,\eta/2)\big)\leq \nu\big(B(x,\eta)\big)\leq 2 \mu\big( B(x,2\eta)\big)= 2\mu(\{x\})$$
(this is possible by applying Lemma \ref{LEMTOPO1} with $\alpha=\eta/2$ and $\beta=\inf_{x\in\supp(\mu)}\mu(\{x\})$).
\item for any $x\in\supp(\mu)$, 
\begin{eqnarray*}
\nu\big(B(x,\eta)\big)&\geq&\mu\big(B(x,\eta/2)\big)-\frac{I_\mu(2r,q)}{N}\\
&\geq&\mu(\{x\})-\frac{I_\mu(2r,q)}{N}
\end{eqnarray*}
where $N$ denotes the cardinal of the support of $\mu$. What is important here is that,
setting $Y=\bigcup_{x\in\supp(\mu)}B(x,\eta)$, one gets
$$\nu(Y)\geq 1-I_\mu(2r,q).$$
\item for any $x\in\supp(\mu)$ and any $z\in B(x,\eta)$, 
$$\nu\big(B(z,r)\big)\leq 2\mu\big(B(z,t)\big)=2\mu\big(B(x,t)\big)$$
$$\nu\big(B(z,2r)\big)\geq\frac 12\mu\big(B(z,t)\big)=\frac12\mu\big(B(x,t)\big).$$
Again, this is possible thanks to Lemma \ref{LEMTOPO1} and because $\mu$ has finite support.
\end{itemize}
Hence, on the one hand, we get
\begin{eqnarray*}
I_\nu(r,q)&=&\sum_{x\in\supp(\mu)} \int_{B(x,\eta)} \nu\big(B(z,r)\big)^{q-1}d\nu(z)+\int_{Y^c}\nu\big(B(z,r)\big)^{q-1}d\nu(z)\\
&\leq&2^{q}\sum_{x\in\supp(\mu)}\mu\big(B(x,t)\big)^{q-1}\mu(\{x\})+\nu(Y^c)\\
&\leq&2^{q}I_\mu(t,q)+I_\mu(2r,q)\\
&\leq&(2^q+1)I_\mu(2r,q).
\end{eqnarray*}
On the other hand, we also have
\begin{eqnarray*}
I_\nu(2r,q)&\geq&\sum_{x\in\supp(\mu)}\int_{B(x,\eta)}\nu\big(B(z,2r)\big)^{q-1}d\nu(z)\\
&\geq&2^{-q}\sum_{x\in\supp(\mu)}\mu\big(B(x,t)\big)^{q-1}\mu(\{x\})\\
&\geq&2^{-q}I_\mu(t,q)\geq2^{-q}I_\mu(r,q).
\end{eqnarray*}
\end{proof}

When $q<1$, the situation is more difficult. Indeed, for an arbitrary close to $\mu$ measure $\nu\in\pk$, $\nu\big(B(x,r))^{q-1}$ may be very large even if
$x\notin\supp(\mu)$.  However, adding an assumption to avoid singularities, we are able to prove the following corollary.
\begin{corollary}\label{CORTOPO3}
 Let $q<1$. There exists $C_q>0$ such that, for any $\mu\in\fk$, for any $r>0$ with $\mu\big(B(x,r)\big)>0$ for any $x\in K$,
there exists $\delta>0$ such that, for any $\nu\in\pk$ with $L(\mu,\nu)<\delta$,
$$I_\nu(2r,q)\leq C_q I_\mu(r,q)\textrm{ and }I_{\nu}(r,q)\geq C_q^{-1} I_{\mu}(2r,q).$$
\end{corollary}
\begin{proof}
 Let $a=\min_{x\in K}\mu(B(x,r))>0$. $a$ is positive since $\mu$ has finite support so that $\mu(B(x,r))$ can only take a finite number of values.
As before, let $t\in(r,2r)$ and let $\eta>0$ be such that
$$\left\{
\begin{array}{l}
 [t-\eta,t+\eta]\cap \big\{\|x-y\|;\ x,y\in\supp(\mu)\big\}=\varnothing;\\
\forall x,y\in\supp(\mu),\ x\neq y\implies \|x-y\|>2\eta.
\end{array}
\right.$$
Then, for any $x\in\supp(\mu)$ and any $z\in B(x,\eta)$, one gets
$$\mu\big(B(x,t)\big)=\mu\big(B(z,t)\big).$$
Now, provided $\delta>0$ is small enough,
Lemma \ref{LEMTOPO1} tells us that, for any $\nu\in\pk$ with $L(\mu,\nu)<\delta$,
$$\forall z\in K,\ \nu\big(B(z,2r)\big)\geq \frac12\mu\big(B(z,t)\big)\geq\frac14\nu\big(B(z,r)\big)$$
(recall that $\mu\big(B(z,t)\big)\geq a>0$). Let $F=K\backslash \bigcup_{x\in\supp(\mu)}B(x,\eta)$. Then
\begin{eqnarray*}
I_\nu(2r,q)&=&\sum_{x\in\supp(\mu)}\int_{B(x,\eta)}\frac{d\nu(z)}{\nu\big(B(z,2r)\big)^{1-q}}+\int_F\frac{d\nu(z)}{\nu\big(B(z,2r)\big)^{1-q}}\\
&\leq&2^{1-q}\sum_{x\in\supp(\mu)}\int_{B(x,\eta)}\frac {d\nu(z)}{\mu\big(B(z,t)\big)^{1-q}}+2^{1-q}\int_F\frac{d\nu(z)}{\mu\big(B(z,t)\big)^{1-q}}\\
&\leq&2^{1-q}\sum_{x\in\supp(\mu)}\frac{\nu\big(B(x,\eta)\big)}{\mu\big(B(x,t)\big)^{1-q}}+2^{1-q}a^{q-1}\nu(F).
\end{eqnarray*}
We apply Lemma \ref{LEMTOPO1} again to observe that, provided $\delta$ is small enough, for any $x\in\supp(\mu)$, 
$$\nu\big(B(x,\eta)\big)\leq 2\mu\big(B(x,3\eta/2)\big)=2\mu(\{x\})$$
and that 
$$\nu(F)\leq \mu\big(F(\eta/2)\big)+\veps=\veps,$$
where $\veps>0$ is arbitrary. Thus we conclude that
\begin{eqnarray*}
 I_\nu(2r,q)&\leq&2^{2-q}I_\mu(t,q)+2^{1-q}a^{q-1}\veps\\
&\leq&2^{3-q}I_\mu(r,q)
\end{eqnarray*}
if $\veps>0$ is small enough. For the other inequality, we simply write
\begin{eqnarray*}
I_\nu(r,q)&\geq&\sum_{x\in\supp\mu}\int_{B(x,\eta)}\frac{d\nu(z)}{\nu\big(B(z,r)\big)^{1-q}}\\
&\geq&\sum_{x\in\supp\mu}2^{q-1}\int_{B(x,\eta)}\frac {d\nu(z)}{\mu\big(B(x,t)\big)^{1-q}}\\
&\geq&\sum_{x\in\supp\mu}2^{q-1}\frac{\nu\big(B(x,\eta)\big)}{\mu\big(B(x,t)\big)^{1-q}}.
\end{eqnarray*}
Now, we can choose $\delta>0$ to ensure that $\nu\big(B(x,\eta)\big)\geq 2^{-1}\mu\big(B(x,\eta/2)\big)=2^{-1}\mu\big(\{x\}\big),$ so that
$$I_\nu(r,q)\geq 2^{q-2}I_{\mu}(t,q)\geq 2^{q-2}I_\mu(2r,q).$$
\end{proof}

\begin{remark}
 The assumption $\mu\big(B(x,r)\big)>0$ for any $x\in K$ is needed only to prove $I_\nu(2r,q)\leq C_q I_\mu(r,q)$.
\end{remark}

Of course, there is a similar statement for $q=1$. Nevertheless, we will need later a more precise result, because
of the logarithm. We restrict ourselves to one inequality.

\begin{corollary}\label{CORTOPOL1}
 Let $\veps>0$, $r>0$ and $\mu\in\fk$. There exists $\delta>0$ such that
for any $\nu\in\pk$ with $L(\mu,\nu)<\delta$, 
$$I_\nu(r,1)\leq \log 2+(1-\veps)I_\mu(2r,1).$$
\end{corollary}
\begin{proof}
 Again, let $t\in(r,2r)$ and let $\eta>0$ be such that
$$\left\{
\begin{array}{l}
 [t-\eta,t+\eta]\cap \big\{\|x-y\|;\ x,y\in\supp(\mu)\big\}=\varnothing;\\
\forall x,y\in\supp(\mu),\ x\neq y\implies \|x-y\|>2\eta.
\end{array}
\right.$$
Then, for any $x\in\supp(\mu)$ and any $z\in B(x,\eta)$, one gets
$$\mu\big(B(x,t)\big)=\mu\big(B(z,t)\big).$$ 
Let $\delta>0$ be such that, for any $x\in\supp(\mu)$ and any $z\in B(x,\eta)$, for any
$\nu\in\pk$ with $L(\mu,\nu)<\delta$, 
$$\nu\big(B(z,r)\big)\leq 2\mu\big(B(z,t)\big)=2\mu\big(B(x,t)\big)\leq 2\mu\big(B(x,2r)\big),$$
$$\nu\big(B(x,\eta)\big)\geq(1-\veps)\mu\big(B(x,\eta/2)\big)=(1-\veps)\mu(\{x\}).$$
Then,
\begin{eqnarray*}
 I_\nu(r,1)&\leq&\sum_{x\in\supp(\mu)}\int_{B(x,\eta)}\log\big(\nu\big(B(z,r)\big)\big)d\nu(z)\\
&\leq&\sum_{x\in\supp(\mu)}\log(2)\nu\big(B(x,\eta)\big)+\sum_{x\in\supp(\mu)}\log \mu\big(B(x,2r)\big)\nu\big(B(x,\eta)\big)\\
&\leq&\log 2+(1-\veps)\sum_{x\in\supp(\mu)}\log \mu\big(B(x,2r)\big)\mu\big(\{x\}\big)\\
&\leq&\log 2+(1-\veps)I_\mu(2r,1).
\end{eqnarray*}
\end{proof}

Another application of Lemma \ref{LEMTOPO1} is the following result on open subsets of $\pk$:
\begin{lemma}\label{LEMTOPOOPEN}
\begin{enumerate}[(a)]
\item  Let $x\in K$, $a\in\mathbb R$ and $r>0$. Then $\{\mu\in\pk;\ \mu\big(B(x,r)\big)>a\}$ is open.
\item Let $E\subset K$ be such that there exists $\alpha>0$ with $E(\alpha)\cap K=E$. 
Then $\{\mu\in\mathcal P(K);\ \mu(E)>0\}$ is open. 
\end{enumerate}
\end{lemma}
\begin{proof}
\begin{enumerate}[(a)]
\item  If $a$ does not belong to $[0,1)$, then the set is either empty or equal to $\pk$. Otherwise, let $\mu\in\pk$ be such that $\mu\big(B(x,r)\big)>a$. One may find $\veps>0$ such that 
$\mu\big(B(x,(1-\veps)r)\big)>a$. Thus the result follows from Lemma \ref{LEMTOPO1} applied with $E=B(x,(1-\veps)r)$,
$\alpha=\veps r$ and $\beta=\big(\mu\big(B(x,(1-\veps)r)\big)-a\big)/2$.
\item The proof is similar (and even easier).
\end{enumerate}
\end{proof}

Finally, we will need that some subsets of $\pk$ are dense in $\pk$. We first recall a result which can be found e.g. in \cite[Lemma 2.2.4.]{OL05}.
\begin{lemma}\label{LEMDENS1}
 Let $(x_i)_{i\geq 1}$ be a dense sequence of $K$. Let $(\mu_{n,i})_{n\geq 1,\ 1\leq i\leq n}$ be a sequence of $\mathcal P(K)$ such that, for any $n\geq 1$ and any $1\leq i\leq n$, 
$\supp(\mu_{n,i})\subset K\cap B(x_i,1/n)$. Then, for any $m\geq 1$, $\bigcup_{n\geq m}\big\{\sum_{i=1}^n p_i\mu_{n,i};\ p_i\geq 0,\ \sum_i p_i=1\big\}$
is dense in $\pk$.
\end{lemma}
Since a compact set is separable, the previous lemma yields in particular that the set $\mathcal F(K)$ of probability measures on $K$ with finite support
is dense in $\pk$. Moreover, taking the $p_i$ in $\mathbb Q$, we can always consider a sequence $(\mu_n)$ of $\mathcal F(K)$ which is dense in $\pk$.

We shall also need several times the following result.
\begin{lemma}\label{LEMDENS2}
 Let $n\geq 1$ and let $K_1,\dots,K_n$ be nonempty subsets of $K$. Then $\big\{\mu\in\pk;\ \mu(K_j)>0\textrm{ for any }j=1,\dots,n\big\}$ is dense in $\pk$.
\end{lemma}
\begin{proof}
 Let $\nu\in\pk$, $\veps>0$ and for each $j=1,\dots,n$, let $x_j\in K_j$. We set
$$\mu=(1-\veps)\nu+\frac{\veps}n\sum_{j=1}^n \delta_{x_j}.$$
Then $L(\mu,\nu)\leq 2\veps$ and $\mu(K_j)>0$ for any $j=1,\dots,n$.
\end{proof}

\subsection{Graph theory}
Surprizingly enough, our constructions of measures with prescribed properties need some results from graph theory. Let us first recall some terminology, which can be found e.g. in \cite{DIESGRAPH}.
Let $G=(V,E)$ be a graph. We recall that two vertices $v,w$ of $E$ are \emph{adjacent} or \emph{neighbour} if $vw$ is an edge of $G$. A \emph{path} between two vertices $v$ and $w$ 
is a sequence of vertices $v_0=v,v_1,\dots,v_N=w$ such that $v_iv_{i+1}$ belongs to the set of edges $E$ for any $i\in\{0,\dots,N-1\}$. We say that the graph is \emph{connected}
provided any two vertices of $G$ can be linked by a path in $G$.

If $U$ is a subset of $V$, the induced subgraph $G(U)$ is the graph whose set of vertices is $U$ and whose set of edges is the subset of $E$ containing all edges $vw$ with $v,w\in U$.

\smallskip

Our first lemma is a well known result in graph theory. For convenience, we provide a proof.

\begin{lemma}\label{LEMGRAPH1}
 Let $G=(V,E)$ be a connected graph. There exists a vertex $v\in V$ such that $G(V\backslash \{v\})$ remains connected.
\end{lemma}
\begin{proof}
 Let $[v_0,\dots,v_N]$ be a path such that its vertices are pairwise different, and whose length is maximal in the set of such paths. 
Then $G\backslash\{v_0\}$ is connected. Indeed, take any $v,w\in V\backslash\{v_0\}$ and let $[w_0,\dots,w_p]$ be the shortest path from $v=w_0$
to $w=w_p$ in $G$. In particular, $w_i=v_0$ for at most one $i\in\{1,\dots,p-1\}$.

If for any $i\in\{1,\dots,p-1\}$, $w_i$ is not equal to $v_0$, there is nothing to prove : the path stays already in $G(V\backslash\{v_0\})$.
Thus, suppose that $w_i=v_0$ for some $i$. If $w_{i-1}$ or $w_{i+1}$ did not belong to $\{v_1,\dots,v_N\}$, then we could add the edge $w_{i-1}v_0$
or $w_{i+1}v_0$ to the path $[v_0,\dots,v_N]$, with a new vertex. This would contradict the maximality of this path. Thus, $w_{i-1}$ and $w_{i+1}$
both belong to $\{v_1,\dots,v_N\}$, and we can replace in $[w_0,\dots,w_p]$ the subpath $[w_{i-1},w_i,w_{i+1}]$ by the corresponding
subpath joining $w_{i-1}$ to $w_{i+1}$ using only vertices in $\{v_1,\dots,v_N\}$, and thus avoiding $v_0$.
\end{proof}

This lemma allows us to give weights on the vertices of a connected graph with particular properties.
\begin{lemma}\label{LEMGRAPH2}
 Let $G=(V,E)$ be a connected graph, let $\veps>0$ and let $\rho\geq 1$. There exist $v_0\in V$ and positive real numbers $(\omega_v)_{v\in V}$
such that
\begin{itemize}
 \item for any $v\in V\backslash\{v_0\}$, one can find a neighbour $u$ of $v$ such that $\omega_v\leq\veps \omega_u^\rho$;
\item $\omega_{v_0}>1/2$;
\item $\sum_{v\in V}\omega_v=1$.
\end{itemize}
\end{lemma}
\begin{proof}
 We proceed by induction on the number of vertices of $G$, the case where this number equals 1 or 2 being trivial. So, suppose that
$V$ contains $n\geq 3$ elements and let $v_1$ be given by Lemma \ref{LEMGRAPH1}. The graph $G(V\backslash\{v_1\})$ remaining connected,
we can apply the induction hypothesis with some $\veps_0\in(0,\veps)$ to the graph $G(V\backslash\{v_1\})$. We get positive numbers $(\theta_v)_{v\in V\backslash\{v_1\}}$
and $v_0\in V\backslash\{v_1\}$. Let $\alpha>0$ be very small and let
$$\left\{
\begin{array}{rcll}
\omega_v&=&(1-\alpha)\theta_v&\textrm{ if }v\neq v_1\\
\omega_{v_1}&=&\alpha.
\end{array}
\right.$$
Provided $\alpha>0$ is small enough, $\omega_{v_1}$ is much smaller that $\omega_v$, for $v$ any neighbour of $v_1$. Moreover, when
$v\in V\backslash\{v_0,v_1\}$, one can find another vertex $u$ which is a neighbour of $v$ in $G(V\backslash\{v_1\})$, in particular in $G$,
such that
$$\theta_v\leq\veps_0\theta_u^\rho.$$
This leads to 
$$\omega_v\leq\frac{\veps_0}{(1-\alpha)^{\rho-1}}\omega_u^\rho\leq \veps\omega_u^\rho,$$
provided $\alpha>0$ is small enough again. The property $\omega_{v_0}>1/2$ is also true, under the same restriction on $\alpha$.
\end{proof}

\subsection{Measures with prescribed properties}
Our results will depend on the construction of probability measures on $K$ having prescribed properties depending on the dimension of $K$.
Of course, the choice of the definition of the dimension will determine the properties of the measure on $K$ that we can expect.
In this direction, the most famous result is the Frostman lemma (see \cite{Fal03}) which is based on the Hausdorff dimension.
\begin{lemma}
 Let $K$ be a nonempty compact subset of $\mathbb R^d$. Then, for every $0\leq t<\dim_{\mathcal H}(K)$, there exists $\mu\in\pk$ and constants
$C>0$, $r_0>0$ such that $\mu\big(B(x,r)\big)\leq Cr^t$ for all $x\in K$ and $0<r\leq r_0$.
\end{lemma}
This lemma has a counterpart for the upper-box dimension, due to Tricot (\cite{Tri82}).
\begin{lemma}\label{LEMANTIFROSTMAN}
Let $K$ be a nonempty compact subset of $\mathbb R^d$. Then, for every $t>\dboxsup(K)$, there exists $\mu\in\pk$
and constants $C>0,r_0>0$ such that $\mu\big(B(x,r)\big)\geq C r^t$ for all $x\in K$ and $0<r\leq r_0$.
\end{lemma}
One can also ask to weaken the assumption in the anti-Frostman lemma, by the use of the lower-box dimension instead of the upper-box dimension.
This is possible, even if there is a price to pay: the conclusion is also weaker, since it is only obtained for one value of $r$, arbitrarily small.
\begin{lemma}\label{LEMANTIFROSTMANLOWERBOX}
Let $K$ be a nonempty compact subset of $\mathbb R^d$. Then, for every $t>\dboxinf(K)$, for every $\alpha>0$, there exists $\mu\in\pk$
and $r\in(0,\alpha)$ such that $\mu\big(B(x,r)\big)\geq 2^{-t} r^t$ for all $x\in K$.
\end{lemma}
\begin{proof}
 One can find $r\in(0,\alpha)$ such that $\mathbf P_{r/2}(K)\leq 2^tr^{-t}.$ We set $P=\mathbf P_{r/2}(K)$ and let $x_1,\dots,x_P\in K$ be centers of disjoint balls
of radius $r/2$. We define
$$\mu=\frac 1P\sum_{j=1}^P \delta_{x_j}.$$
Then, for any $x\in K$, one can find $i\in\{1,\dots,P\}$ such that $x_i\in B(x,r)$ (otherwise $B(x,r/2)$ would not intersect
any $B(x_j,r/2)$). Thus
$$\mu\big(B(x,r)\big)\geq \frac1P\geq 2^{-t}r^t.$$
\end{proof}

We can, like in Frostman lemma, obtain measures with an upper estimate of the size of balls using the upper-box dimension only. However, here too, we have
to work with a fixed radius.
\begin{lemma}\label{LEMFROSTMANUPPERBOX}
Let $K$ be a nonempty compact subset of $\mathbb R^d$ and let $0\leq t<\dboxsup(K)$. Then, for any $\alpha>0$, one can find $r\in(0,\alpha)$
and $\mu\in\pk$ such that, for any $x\in K$, $\mu\big(B(x,r)\big)\leq r^t$.
\end{lemma}
\begin{proof}
 The proof is similar to that of Lemma \ref{LEMANTIFROSTMANLOWERBOX}. By definition of the upper-box dimension, one can find $r\in(0,\alpha)$ such that 
$$r^{-t}\leq \mathbf P_{r}(K).$$
 Let $P=\mathbf P_{r}(K)$ and let 
$x_1,\dots,x_P\in K$ be the centers of disjoint balls with radius $r$. Then the measure
$$\mu=\frac 1P\sum_{j=1}^P \delta_{x_j}$$
is convenient. Indeed, for any $x\in K$, the ball $B(x,r)$ contains at most one of the $x_j$, so that $\mu\big(B(x,r)\big)\leq P^{-1}\leq r^t$.
\end{proof}

\medskip

We finally use the box separation index. This is much more delicate, and we work directly with $I_{\mu}(r,q)$ instead of $\mu\big(B(x,r)\big)$.
\begin{lemma}\label{LEMIMUBSI}
 Let $K$ be a nonempty compact subset of $\mathbb R^d$, let $n\geq 1$ and let $q\leq 0$. There exists $\nu_n\in\pk$ with finite support
such that, for any $\mu=\sum_{i=1}^N p_i\delta_{x_i}\in\pk$ with finite support, for any $\theta\in(0,1)$, setting
$\nu=(1-\theta)\mu+\theta\nu_n$, then 
$$I_\nu(3.2^{-n},q)\leq N(1-\theta)^qa^q+\theta^q+\frac{\theta^q}{2^q}C_n(K)^{1-q},$$
where $a=\min_i(p_i)$. Moreover, $\nu\big(B(x,3.2^{-n})\big)>0$ for any $x\in K$.
\end{lemma}
\begin{proof}
 We use the notations of the introduction. We decompose $K_n$ into its connected components $\mathcal C_{n,1}(K),\dots,\mathcal C_{n,C_n(K)}(K)$. 
For any $j=1,\dots,C_n(K)$, we denote by $C_j^i$, $i=1,\dots,\kappa_j,$ the half-closed dyadic cubes of size $2^{-n}$ appearing in $\mathcal C_j$. 
By definition, $C_j^i\cap K$ is never empty, and we can fix $y_j^i$ a point in $C_j^i\cap K$.

We endow each $\mathcal C_{n,j}$ with a graph structure. Precisely, we say that two cubes $C_j^i$ and $C_j^l$ are adjacent provided
their closure do intersect. The graph $\mathcal C_{n,j}$ is then a connected graph. Moreover, it is easy to check that if $C_j^i$ and $C_j^l$
are adjacent, then for any $x\in C_j^i$, $C_j^l$ is contained in $B(x,3.2^{-n})$ (at this part of the proof, we find convenient to use
the $\ell^\infty$-norm on $\mathbb R^d$, and to suppose that $B(x,r)$ is the ball corresponding to this norm. Working with a different norm
would only replace here 3 by a constant depending on $d$. This would not change at all the content of our main theorem.)

Let $\veps>0$ be such that 
$\veps(\kappa_1+\dots+\kappa_{C_n(K)})\leq C_n(K)^q.$
We apply Lemma \ref{LEMGRAPH2} to each connected graph $\mathcal C_{n,j}$ with $\rho=1-q$. This gives weights $\omega_j^i$, $i=1,\dots,\kappa_j$. Renumbering the cubes if necessary,
we can always assume that, for any $i$ in $1,\dots,\kappa_j-1$, there exists $l\neq i$ such that the cubes $C_j^i$ and $C_j^l$
are adjacent, and $\omega_j^i\leq \veps(\omega_j^l)^{1-q}$ (that is we suppose that the vertex $v_0$ in Lemma \ref{LEMGRAPH2} is $C_j^{\kappa_j}$.)

We finally set
\begin{eqnarray*}
 \lambda_j&=&\sum_{i=1}^{\kappa_j}\omega_j^i \delta_{y_j^i}\\
\nu_n&=&\sum_{j=1}^{C_n(K)}\frac1{C_n(K)}\lambda_j.
\end{eqnarray*}
Let now $\mu$ and $\nu$ as in the assumptions of the lemma. We first observe that for any $x\in K$, $\nu\big(B(x,3.2^{-n})\big)>0$. Indeed,
any $x\in K$ belongs to some $C_j^i$ and $\nu\big(B(x,3.2^{-n})\big)\geq \nu(\{y_j^i\})>0$. Let us now control $I_{\nu}(3.2^{-n},q)$:
$$I_{\nu}(3.2^{-n},q)=\sum_{j=1}^{C_n(K)}\sum_{i=1}^{\kappa_j}\int_{C_j^i\cap K}\frac{d\nu(x)}{\nu\big(B(x,3.2^{-n})\big)}.$$
To estimate the integral, we distinguish three different cases:
\begin{enumerate}[(a)]
 \item There exists (at least) one point in the support of $\mu$ which lies in $C_j^i\cap K$.
Then we just use that, for any $x\in C_j^i\cap K$, $\nu\big(B(x,3.2^{-n})\big)\geq \nu(C_{j}^i\cap K)\geq(1-\theta) a$, so that
$$\int_{C_j^i\cap K}\frac{d\nu(x)}{\nu\big(B(x,3.2^{-n})\big)^{1-q}}\leq \nu(C_j^i\cap K)^q\leq (1-\theta)^qa^q.$$
Observe that this estimate concerns at most $N$ cubes.
\item $\supp(\mu)\cap C_{j}^i$ is empty, and $i\leq \kappa_j-1$. Let $C_j^l$ be a neighbour of $C_j^i$ such that $\omega_j^i\leq \veps(\omega_j^l)^{1-q}$.
For any $x\in C_j^i\cap K$, we already observed that $B(x,3.2^{-n})$ contains $C_j^l\cap K$, so that $\nu\big(B(x,3.2^{-n})\big)\geq \nu(C_j^l\cap K)\geq 
\frac{\theta}{C_n(K)}\omega_j^l$.
Thus,
$$\int_{C_j^i\cap K}\frac{d\nu(x)}{\nu\big(B(x,3.2^{-n})\big)^{1-q}}\leq \frac{\nu(C_j^i\cap K)}{\theta^{1-q}(\omega_j^l)^{1-q}C_n(K)^{q-1}}\leq \frac{\theta^q \omega_j^i}{(\omega_j^l)^{1-q}C_n(K)^{q}}\leq \frac{\theta^q\veps}{C_n(K)^q}.$$
\item For the last cubes, namely $\supp(\mu)\cap C_j^i$ is empty and $i=\kappa_j$, we argue as in case (a),
except that now the mass of the cube $C_j^i\cap K$ is given by $\nu_n$:
\begin{eqnarray*}
 \int_{C_j^i\cap K}\frac{d\nu(x)}{\nu\big(B(x,3.2^{-n})\big)^{1-q}}&\leq&\nu(C_j^i\cap K)^q\\
&\leq&\frac{\theta^q}{C_n(K)^q}(\omega_j^{\kappa_j})^q\\
&\leq&\left(\frac{\theta}2\right)^qC_n(K)^{-q}.
\end{eqnarray*}
\end{enumerate}
Putting this together, and remembering the value of $\veps$, we get that 
\begin{eqnarray}
 \nonumber I_\nu(3.2^{-n},q)&\leq&N(1-\theta)^q a^q+\sum_{j=1}^{C_n(K)}\sum_{i=1}^{\kappa_j-1}\frac{\theta^q\veps}{C_n(K)^q}+\sum_{j=1}^{C_n(K)}\left(\frac{\theta}2\right)^qC_n(K)^{-q}\\
\label{EQLEMIMUBSI} &\leq&N(1-\theta)^qa^q+\theta^q+\frac{\theta^q}{2^q}C_n(K)^{1-q}.
\end{eqnarray}
\end{proof}

\subsection{Results true for all measures}
The results true for all measures which are quoted in Theorem \ref{THMMAINLQSPECTRUM} are not new. All of them can be found in \cite{OL05} or in \cite{OL08}, except
$\dsupmuq\leq s_\mathcal P$ for $q\geq 1$. But this inequality can be found in \cite{Cut95} for $q=1$, and the remaining cases follow easily from the next lemma, which is an easy application 
of Jensen inequality.

\begin{lemma}\label{LEMJENSEN}
Let $q_1\leq q_2$ be two real numbers and let $\mu\in\mathcal P(K)$. Then 
$$\dinfmu(q_1)\geq \dinfmu(q_2)\textrm{ and }\dsupmu(q_1)\geq\dsupmu(q_2).$$
\end{lemma}

\section{Typical lower $L^q$-dimensions, $q\leq 0$}
Let $K$ be an infinite compact subset of $\mathbb R^d$. In this section, we prove that a typical measure $\mu\in\pk$ satisfies $\dinfmuq=s_{\rm sep}$ for any $q\leq 0$.
We first observe that we can work with a fixed value of $q$. Indeed, suppose that we are able to prove that, for any $q\leq 0$, the set
$$\mathcal M(q)=\big\{\mu\in\pk;\ \dinfmuq=s_{\rm sep}\big\}$$
contains a dense $G_\delta$-set. Then consider $(q_n)$ a dense sequence in $(-\infty,0]$,with $q_0=0$ and let
$$\mathcal M=\bigcap_n\mathcal M(q_n)$$
which itself contains a dense $G_\delta$-set. Let $\mu\in\mathcal M$ and let $q\in(-\infty,0]$. One can find two subsequences $(q_{\phi(n)})$ and $(q_{\psi(n)})$
such that $q_{\phi(n)}\leq q\leq q_{\psi(n)}$. Since $\dinfmu(q_{\psi(n)})\leq \dinfmuq\leq \dinfmu(q_{\phi(n)})$, taking the limit, we get $\dinfmuq=s_{\rm sep}$.

This trick works in any case. In particular, in this section, from now on, we fix $q\leq 0$.
\subsection{The lower estimate}\label{SECLEQLEQ0} We first prove that a typical measure $\mu\in\pk$ satisfies $\dinfmuq\geq s_{\rm sep}$. 
Let $n\geq 1$, and recall that $\mathcal C_{n,1},\dots,\mathcal C_{n,C_n(K)}$ denote the connected components of $K_n$, the union
of the dyadic cubes of size $2^{-n}$ intersecting $K$. We also define
$$\mathcal A_{n,j}=\bigcup_{C\in\mathcal C_{n,j}}C\cap K.$$
It is important to notice that $\textrm{dist}(\mathcal A_{n,j},\mathcal A_{n,k})\geq 2^{-n}$ provided $j\neq k$, and that $K$ is the (disjoint) union
of the $\mathcal A_{n,j}$. In particular, $\mathcal A_{n,j}(2^{-n-1})\cap K=\mathcal A_{n,j}$. We finally set
\begin{eqnarray*}
 \mathcal R_n&=&\big\{\mu\in\pk;\ \forall j\in\{1,\dots,C_n(K)\},\ \mu(\mathcal A_{n,j})>0\big\}\\
\mathcal R&=&\bigcap_{n\geq 1}\mathcal R_n,
\end{eqnarray*}
and we claim that $\mathcal R$ is the dense $G_\delta$-set we are looking for. Indeed, each $\mathcal R_n$ is dense by Lemma \ref{LEMDENS2} and open by Lemma \ref{LEMTOPOOPEN}. Moreover, any $\mu$ in $\mathcal R$ has its lower $L^q$-dimension greater that or equal to $s_{\rm sep}$. Indeed, let $r>0$ be small, and let $n\geq 1$
be such that $2^{-(n+1)}< r\leq 2^{-n}$. We know that
\begin{eqnarray*}
 I_{\mu}(r,q)&=&\sum_{j=1}^{C_n(K)}\int_{\mathcal A_{n,j}}\frac{d\mu(x)}{\mu\big(B(x,r)\big)^{1-q}}\\
&\geq&\sum_{j=1}^{C_n(K)}\int_{\mathcal A_{n,j}}\frac{d\mu(x)}{\mu(\mathcal A_{n,j})^{1-q}},
\end{eqnarray*}
because the $\mathcal A_{n,j}$ are well separated, so that $\mathcal A_{n,j}\supset B(x,r)\cap K$ for any $x\in \mathcal A_{n,j}$ and any $r\leq 2^{-n}$. Thus we find
$$I_\mu(r,q)\geq \sum_{j=1}^{C_n(K)}\mu(\mathcal A_{n,j})^{q}.$$
Now, since $\sum_{j=1}^{C_n(K)}\mu(\mathcal A_{n,j})=1$, it is well known that this quantity is minimal when $\mu(\mathcal A_{n,j})=\frac1{C_n(K)}$
for any $j=1,\dots,C_n(K)$. This shows that
$$I_\mu(r,q)\geq C_n(K)^{1-q}.$$
Taking the logarithm, dividing by $(q-1)\log r$, which is positive, and taking the liminf, we get the aforementioned lower bound for $\dinfmuq$.

\subsection{The upper estimate}\label{SECUPPERBSI}
We now prove that a typical measure $\mu\in\pk$ satisfies $\dinfmuq\leq s_{\rm sep}$. It is sufficient to prove that, for any $t>s_{\rm sep}$, a typical measure
$\mu\in\pk$ satisfies $\dinfmuq\leq t$.

Let $\mu\in\mathcal F(K)$, $\mu=\sum_{i=1}^N p_i\delta_{x_i}$, $a=\min_i(p_i)>0$ and let $l\geq 1$. By definition of the box separation index, one can find $n\geq l$ very large such that 
$$N(1-N^{-1})^qa^q+N^{-q}+2^{-q}N^{-q}C_n(K)^{1-q} \leq 2^{nt(1-q)}.$$

Let $\nu_n$ be the measure given by Lemma \ref{LEMIMUBSI} and let 
$$\nu_{\mu,l}=\left(1-\frac1N\right)\mu+\frac 1N\nu_n.$$
We can apply Corollary \ref{CORTOPO3} to the measure $\nu_{\mu,l}$ and to $r=3.2^{-n}$ to get some positive $\delta_{\mu,l}$ such that, for any $\nu$ with $L(\nu,\nu_{\mu,l})<\delta_{\mu,l}$, 
\begin{eqnarray*}
 I_\nu(6.2^{-n},q)&\leq&C_qI_{\nu_{\mu,l}}(3.2^{-n},q)\\
&\leq&C_q\left(N(1-N^{-1})^qa^q+N^{-q}+2^{-q}N^{-q}C_n(K)^{1-q}\right)\\
&\leq&C_{q}2^{nt(1-q)}.
\end{eqnarray*}
We finally set
$$\mathcal R=\bigcap_{l\geq 1}\bigcup_{\mu\in\mathcal F(K)}B_{L}(\nu_{\mu,l},\delta_{\mu,l}).$$
For any fixed $l\geq 1$, $\{\nu_{\mu,l};\ \mu\in\mathcal F(K)\}$ is dense in $\pk$. Indeed, let $\lambda\in\pk$ and let $\veps>0$. Since $K$ is infinite, there
exists $\mu\in\fk$ such that $L(\mu,\lambda)<\veps/2$ and the cardinal number of the support of $\mu$, denoted by $N$, satisfies $N>4/\veps$. Now, 
$$L(\nu_{\mu,l},\lambda)<\frac\veps2+\frac 2N<\veps.$$
 Thus, $\mathcal R$ is a dense $G_\delta$-set. Moreover,
for any $\nu\in\mathcal R$, there exist arbitrarily large integers $n$ such that 
$$I_\nu(6.2^{-n},q)\leq C_{q}2^{nt(1-q)}.$$
Taking the logarithm and dividing by $(q-1)\log(6.2^{-n})$, this shows that $\dinfnuq\leq t$.

\subsection{The box separation index of self-similar compact sets}
We end up this section by computing the box separation index of a self-similar compact set. Fix an integer $N\geq 2$ and let $S_i:\mathbb R^d\to\mathbb R^d$,
$i=1,\dots,N$ be contracting similarities with respective ratio $r_i$. Let $K$ be the self-similar compact set associated to these similarities, namely
$K$ is the unique nonempty compact subset of $\mathbb R^d$ satisfying
$$K=\bigcup_i S_i(K).$$
We say that $(S_1,\dots,S_N)$ satisfies the \emph{open set condition} if there exists an open nonempty and bounded subset $\mathcal U$ of $\mathbb R^d$
with $S_i\mathcal U\subset\mathcal U$ and $S_i\mathcal U\cap S_j\mathcal U=\varnothing$ for all $i,j$ with $i\neq j$. 
We also say that $(S_1,\dots,S_N)$ satisfies the \emph{strong separation condition} if $S_i K\cap S_j K=\varnothing$ for all $i\neq j$.

\medskip

When $(S_1,\dots,S_N)$ satisfies the open set condition, the Hausdorff dimension and the box dimension of $K$ are well known. Define $\beta$ as the unique solution
of $\sum_{i=1}^N r_i^\beta=1$. Then $\dim_{\mathcal H}(K)=\dboxinf(K)=\dboxsup(K)=\beta$ (all details can be found in \cite{Fal03}). Under a slightly stronger
assumption, this is also the value of the box separation index.

\begin{theorem}
 Let $(S_1,\dots,S_N)$ be contracting similarities of $\mathbb R^d$ with ratios $(r_1,\dots,r_N)$ and let $K$ be the associated self-similar
compact set. If $(S_1,\dots,S_N)$ satisfies the strong separation condition, then 
$$\bsi(K)=\dim_{\mathcal H}(K)=\dboxinf(K)=\dboxsup(K)=\beta$$
where $\beta$ is the unique solution of $\sum_{i=1}^N r_i^\beta=1$.
\end{theorem}
\begin{proof}
 Since $\bsi(K)\leq\dboxinf(K)$, we just need to prove that $\liminf_n \frac{\log C_n(K)}{n\log 2}\geq\beta$. For a word
$\mathbf i=i_1\dots i_p$ with entries in $\{1,\dots,N\}$, define
$$S_{\mathbf i}=S_{i_1}\circ\dots\circ S_{i_p},\ r_{\mathbf i}=r_{i_1}\dots r_{i_p},\ \rho_{\mathbf i}=r_{\mathbf i}^\beta.$$
Since the strong separation condition is satisfied, there exists $m\geq 0$ such that 
$$\textrm{dist}\big(S_i(K),S_j(K)\big)\geq 2^{-m}\textrm{ for any }i\neq j.$$
For $r>0$, denote by $I(r)$ the set of words $\mathbf i$ such that
$$r_{i_1}\dots r_{i_p}\leq r\textrm{ and }r_{i_1}\dots r_{i_{p-1}}>r.$$
Such a word is called \emph{minimal} with respect to $r$. The key points are the following facts which can be found e.g. in 
the proof of \cite[Theorem 9.3]{Fal03}:
\begin{itemize}
 \item For any $r>0$, $\sum_{\mathbf i\in I(r)}\rho_{\mathbf i}=1$;
\item For any $r>0$, $K\subset\bigcup_{i\in I(r)}S_i(K).$
\end{itemize}
Now, by definition of $I(r)$, $\rho_{\mathbf i}\leq r^\beta$ for any $\mathbf i\in I(r)$, so that 
$$\card\big(I(r)\big)\geq r^{-\beta}.$$
Pick now $\mathbf i\neq \mathbf j$ in $I(r)$. By minimality of $\mathbf i$ and $\mathbf j$, $\mathbf i$ is not the beginning
 of $\mathbf j$, and conversely. Let $k$ be the smallest index with $i_k\neq j_k$ and let $ \mathbf u$ be the word
$\mathbf u=i_1\dots i_{k-1}$. Then 
\begin{eqnarray*}
 \textrm{dist}\big(S_{\mathbf i}(K),S_{\mathbf j}(K)\big)&\geq& \textrm{dist}\big(S_{\mathbf u}S_{i_k}(K),S_{\mathbf u}S_{j_k}(K)\big)\\
&\geq&r_{i_1}\dots r_{i_{k-1}}\textrm{dist}\big(S_{i_k}(K),S_{j_k}(K)\big)\\
&\geq&2^{-m}r.
\end{eqnarray*}
We specialize this inequality for $r=2^{-n}$. In each $S_{\mathbf i}(K)$, $\mathbf i\in I(2^{-n})$, one can find a point belonging to $K$. 
Furthermore, the distance between $S_{\mathbf i}(K)$ and $S_{\mathbf j}(K)$ is greater than $2^{-(n+m)}$ when $\mathbf i\neq\mathbf j\in I(2^{-n})$.
Thus, the closure of the dyadic cubes of size $2^{-(n+m+2)}$ intersecting $S_\mathbf i(K)$ and the closure
of the dyadic cubes of size $2^{-(n+m+2)}$ intersecting $S_\mathbf j(K)$, $\mathbf i\neq\mathbf j\in I(2^{-n})$, are disjoint. Hence,
$$C_{n+m+2}(K)\geq\card\big(I(2^{-n})\big)\geq 2^{-n\beta}.$$
Taking the logarithm and the liminf, this yields $\bsi(K)\geq\beta$.
\end{proof}
\begin{remark}
 The example of $K=[0,1]$ shows that we cannot only assume the open set condition.
\end{remark}

\section{Local uniform upper box dimension and typical upper $L^q$-dimension}
In this section, we study the link between the local uniform upper box dimension of the compact set $K$ and the typical upper $L^q$-dimension, $q> 1$,
of the probability measures on $K$. We also show that the local uniform upper box dimension does not coincide with the local upper box or the local lower box dimensions.

From now on, we fix $q> 1$ and set $s_u=\dboxsuplocunif(K)$.

\subsection{The lower estimate}\label{SECUPPERUNIFORMBOXANDLQDIM}
Let $t<s_u$. One intends to show that $\dsupmuq\geq t$ for a typical measure $\mu\in\pk$. Let $(x_i)_{i\geq 1}$ be a dense sequence in $K$ and let $n\geq 1$. By definition
of $s_u$, one may find $r_n<1/n$ such that, for any $i=1,\dots,n$, 
$$\mathbf P_{r_n}\big(K\cap B(x_i,1/n)\big)\geq r_n^{-t}.$$
Arguing like in the proof of Lemma \ref{LEMFROSTMANUPPERBOX}, we can construct for each $n\geq 1$ and each $i\leq n$ a measure $\mu_{n,i}\in\pk$ such that
$\supp(\mu_{n,i})\subset B(x_i,1/n)$ and, for any $x\in K$, $\mu_{n,i}\big(B(x,r_n)\big)\leq r_n^t$. 

Let also $\delta_n>0$ be given by Lemma \ref{LEMTOPO1} for $\alpha=r_n/2$ and $\beta=r_n^t$. We finally set
\begin{eqnarray*}
 \Lambda_n&=&\left\{\sum_{i=1}^n p_i\mu_{n,i};\ p_i\geq 0,\ \sum_{i=1}^n p_i=1\right\}\\
\mathcal R_n&=&\bigcup_{\mu\in\Lambda_n}B_L(\mu,\delta_n)\\
\mathcal R&=&\bigcap_{m\geq 1}\bigcup_{n\geq m}\mathcal R_n.
\end{eqnarray*}
By Lemma \ref{LEMDENS1}, for any $m\geq 1$, $\bigcup_{n\geq m}\mathcal R_n$ is dense, so that $\mathcal R$ is a dense $G_\delta$-subset of $\pk$. Pick now $\nu\in\mathcal R$. 
There exists an arbitrarily large integer $n$ and $\mu\in\Lambda_n$ such that $L(\mu,\nu)<\delta_n$. Now, $\mu$ satisfies, for any $x\in K$, 
$$\mu\big(B(x,r_n)\big)=\sum_{i=1}^n p_i\mu_{n,i}\big(B(x,r_n)\big)\leq r_n^t.$$
Taking into account the value of $\delta_n$, this yields
$$\nu\big(B(x,r_n/2)\big)\leq 2r_n^t.$$
Integrating this inequality, we get
$$I_\nu(r_n/2,q)\leq 2^{q-1}r_n^{t(q-1)}.$$
Since $r_n$ can be taken arbitrarily small, this implies $\dsupnuq\geq t$.

\subsection{The upper estimate}
Let $t>s_u$. Our aim is to show that, generically, $\dsupmuq\leq t$. By definition of the local uniform upper box dimension, 
one can find $x_1,\dots,x_n\in K$, all different, $r_0>0$, with $4r_0<\min_{i\neq j}\|x_i-x_j\|$ and $\alpha>0$,
such that, for any $r\in(0,\alpha)$, one may find $i\in\{1,\dots,n\}$ with 
$$\frac{\log \mathbf N_r\big(K\cap B(x_i,r_0)\big)}{-\log r}\leq t.$$
The crucial point here is a result which appears in the proof of \cite[Lemma 2.3.1.]{OL05}. It is based on Jensen's inequality.
\begin{lemma}\label{LEMJENSENLQ}
 Let $\mu\in\pk$, let $E$ be a Borel subset of $K$ with $\mu(E)>0$. Then, for any $r>0$,
$$I_\mu(2r,q)\geq\frac{\mu(E)^q}{\mathbf N_r(E)^{q-1}}.$$
\end{lemma}
We then set
$$\mathcal R=\big\{\mu\in\pk;\ \forall i\in\{1,\dots,n\},\ \mu\big(B(x_i,r_0)\big)>0\big\}.$$
The set $\mathcal R$ is dense (Lemma \ref{LEMDENS2}) and open (Lemma \ref{LEMTOPOOPEN}). Moreover, for any $\mu\in\mathcal R$ and 
any $r>0$, 
$$I_\mu(2r,q)\geq\sup_{i=1,\dots,n} \frac{\mu\big(B(x_i,r_0)\big)^q}{\mathbf N_r\big(K\cap B(x_i,r_0)\big)^{q-1}}\geq \frac{C}{\inf_{i=1,\dots,n}\big(\mathbf N_r\big(K\cap B(x_i,r_0)\big)^{q-1}},$$
where $C=\inf_{i=1,\dots,n}\mu\big(B(x_i,r_0)\big)^q$.
Taking the logarithm, we find
$$\log I_\mu(2r,q)\geq \log{C}-(q-1)\inf_{i=1,\dots,n}\log \mathbf N_r\big(K\cap B(x_i,r_0)\big).$$
Now, for any $r$ in $(0,\alpha)$, this becomes
$$\log I_\mu(2r,q)\geq \log C+t(q-1)\log r.$$
We now divide this inequality by $(q-1)\log r$ (which is negative) and we take the limsup to obtain that $\dsupmuq\leq t$.

\subsection{Comparison of three notions of dimensions}\label{SECCANTOR}
We now show that the local uniform upper box dimension is not always equal to the (local) lower box dimension or to
the (local) upper box dimension. For simplicity, we set
\begin{eqnarray*}
 s^+&=&\inf_{x\in K}\dboxsuploc(x,K)\\
s_-&=&\inf_{x\in K}\dboxinfloc(x,K).
\end{eqnarray*}
That $s_u>s_-$ for some compact sets is easy: it suffices to take a Cantor set $K$ for which $\dboxinf(K)<\dboxsup(K)$. The uniformity
in the construction of the Cantor set ensures that
$$s_-=\dboxinf(K)\textrm{ and }s_u=s^+=\dboxsup(K)$$
(see \cite{Tri82} for details). 

To prove that it is possible that $s_u<s^+$, we also use Cantor sets. Let $(n_k)$ be a sequence of integers increasing to $+\infty$ with 
$n_0=1$ and $n_{k+1}>5n_k$. We define two sequences $(\alpha_n)$ and $(\beta_n)$ by $\alpha_0=\beta_0=1$  and
$$\alpha_{n+1}=\left\{
\begin{array}{ll}
\displaystyle 27&\textrm{ provided }n\in[n_k,2n_k)\\[0.25cm]
\displaystyle 3&\textrm{ provided }n\in[2n_k,3n_k)\\[0.25cm]
\displaystyle 9&\textrm{ otherwise}
\end{array}\right.$$
$$\beta_{n+1}=\left\{
\begin{array}{ll}
 \displaystyle 27 &\textrm{ provided }n\in[3n_k,4n_k)\\[0.25cm]
\displaystyle 3&\textrm{ provided }n\in[4n_k,5n_k)\\[0.25cm]
\displaystyle  9 &\textrm{ otherwise.}
\end{array}\right.$$

Define $E$ (resp. $F$) a Cantor subset of $[0,1]$
(resp. of $[2,3]$) as follows : $E=\bigcap_n E_n$ (resp. $F=\bigcap_n F_n$) where at each step
we divide each subinterval of $E_n$ (resp. $F_n$) into $\alpha_{n+1}$ (resp. $\beta_{n+1}$) intervals of size 
$100^{-(n+1)}$. Thus $E_n$ (resp. $F_n$) consists of $\alpha_1\times\dots\times\alpha_n$ (resp. $\beta_1\times\dots\times\beta_n$)
intervals of size $100^{-n}$.We finally set $K=E\cup F$. A key point in the construction is that, for each $n$, 
$$\min(\alpha_1\dots\alpha_n,\beta_1\dots\beta_n)=9^n.$$
Thus we need at most $9^n$ intervals of size $100^{-n}$ to cover either $E$ or $F$. On the contrary, when we look separately at $E$ or at $F$, we need 
for certain values of $n$ more intervals. 

\smallskip 

Let us proceed with the details. By uniformity of the construction of each Cantor set, it is not hard to show that
$$\inf_{x\in K}\dboxsuploc(x,K)=\inf\big(\dboxsup(E),\dboxsup(F)\big).$$
Now, 
$$\dboxsup(E)=\limsup_{n\to+\infty}\frac{\log \alpha_1\dots\alpha_n}{n\log100}=\lim_{k\to+\infty}\frac{\log9^{n_k}27^{n_k}}{2n_k\log100}=\frac{5\log 3}{2\log 100},$$
and, in the same way,
$$\dboxsup(F)=\limsup_{n\to+\infty}\frac{\log \beta_1\dots\beta_n}{n\log100}=\lim_{k\to+\infty}\frac{\log9^{3n_k}27^{n_k}}{4n_k\log100}=\frac{9\log 3}{4\log 100}.$$
On the contrary, we shall see that for this compact set $K$, $s_u\leq\frac{2\log 3}{\log 100}$ (actually, $s_u$ is equal to $\frac{2\log 3}{\log 100}$).
Indeed, it is enough to show that 
$$\limsup_{r\to 0}\left(\frac{\inf\big(\log \mathbf N_r(E),\log \mathbf N_r(F)\big)}{-\log r}\right)\leq\frac{2\log 3}{\log 100}.$$
This follows easily from
$$\limsup_{n\to+\infty}\frac{\log 9^n}{n\log 100}=\frac{2\log 3}{\log 100}.$$

\begin{remark}
 The local uniform upper box dimension can also be easily computed for a regular compact set. We recall that a compact set $K\subset\mathbb R^d$
with $\dimh(K)=s$ 
is called \emph{Ahlfors regular} if there are positive constants $c_1,c_2,r_0>0$ such that 
$$c_1(2r)^s\leq \mathcal H^s\big(K\cap B(x,r)\big)\leq c_2 (2r)^s$$
for all $x\in K$ and all $0<r<r_0$.
For a Ahlfors-regular compact set, it is easy (and classical) to check that
$$\inf_{x\in K}\dboxinfloc(x,K)=\dboxsuplocunif(K)=\dimh(K)=s.$$
This happens in particular if $K$ is a self-similar compact set satisfying the open set condition.
\end{remark}

\section{Typical $L^q$-dimensions, the remaining cases, $q\neq 1$}
In this section, we turn to the remaining cases, for $q\neq 1$. More precisely, we prove that a typical measure $\mu\in\pk$ satisfies
\begin{itemize}
 \item $\dsupmuq=\dboxsup(K)$ provided $q\in[0,1)$;
\item $\dsupmuq=+\infty$ provided $q<0$;
\item $\dinfmuq=0$ provided $q\in(0,1)$.
\end{itemize}
As previously, we can work with a fixed $q\in\mathbb R$.

\subsection{Typical upper $L^q$-dimension, $q\in[0,1)$}\label{SECLQUPPERQ01}
In \cite{OL08}, it is proved that any measure $\mu\in\pk$ satisfies $\dsupmuq\leq s$ where $s=\dboxsup(K)$. So, since the result
is trivial if $\dboxsup(K)=0$, we may assume that $s$ is positive, and we have just to prove that, given any $t\in(0,s)$, a typical 
$\mu\in\pk$ satisfies $\dsupmuq\geq t$. 

Let $(\mu_n)$ be a sequence of $\mathcal F(K)$ which is dense in $\pk$. Write $\mu_n=\sum_{i=1}^N p_i\delta_{x_i}$, where the $x_i$
are all different and each $p_i$ is nonzero. Let $\alpha_n>0$ be such that 
$$\alpha_n\leq \frac1n,\ N\alpha_n^t\leq\frac12\textrm{ and }4\alpha_n<\min_{i\neq j}\|x_i-x_j\|.$$
For these values of $t$ and $\alpha_n$, Lemma \ref{LEMFROSTMANUPPERBOX} gives us a real number $r_n\in(0,\alpha_n)$
and a measure $m_n\in\pk$ such that, for any $x\in K$, $m_n\big(B(x,r_n)\big)\leq r_n^t$. We then set
$$\veps_n=\frac1{-\log r_n},\ \nu_n=(1-\veps_n)\mu_n+\veps_n m_n\textrm{ and }Y_n=\bigcup_{i=1}^N B(x_i,r_n).$$
We first observe that 
\begin{eqnarray*}
 \nu_n(Y_n)&\leq&1-\veps_n+\veps_n\sum_{i=1}^N m_n\big(B(x_i,r_n)\big)\\
&\leq&1-\veps_n+\veps_nNr_n^t\\
&\leq&1-\frac{\veps_n}2.
\end{eqnarray*}
Moreover, for any $x\notin Y_n$, $\nu_n\big(B(x,r_n)\big)\leq  \veps_n r_n^t$, so that, 
\begin{eqnarray*}
 I_{\nu_n}(r_n,q)&\geq&\int_{Y_n^c}\frac{d\nu_n(x)}{\nu_n\big(B(x,r_n)\big)^{1-q}}\\
&\geq&\nu_n(Y_n^c)\veps_n^{q-1}r_n^{t(q-1)}\\
&\geq&\frac12 \veps_n^q{r_n^{t(q-1)}}
\end{eqnarray*}
Now, Corollary \ref{CORTOPO3} (see also the remark following it) gives us a real number $\delta_n>0$ such that, for any $\nu\in\mathcal P(K)$
satisfying $L(\nu,\nu_n)<\delta_n$, one has
$$I_{\nu}(r_n/2,q)\geq \frac 14\veps_n^qr_n^{t(q-1)}.$$
Taking the logarithm and dividing by $(q-1)\log r_n$, which is positive, we get
$$\frac{\log I_\nu(r_n/2,q)}{(q-1)\log(r_n)}\geq C_1+C_2\frac{ \log(-\log(r_n))}{\log r_n}+t.$$
Finally, we observe that since $(\veps_n)$ is going to zero, $(\nu_n)$ keeps dense in $\mathcal P(K)$. We then consider the dense $G_\delta$-set
$$\mathcal R=\bigcap_{m\geq 0}\bigcup_{n\geq m}B_L(\nu_n,\delta_n).$$
The work done before implies immediately that any $\nu\in\mathcal R$ satisfies $\dsupnuq\geq t$.  

\subsection{Typical upper $L^q$-dimension, $q<0$}
Let $(\mu_n)$ be a sequence in $\mathcal F(K)$ which is dense in $\pk$. For any $n\in\mathbb N$, let $m_n=\delta_{y_n}$ 
where $y_n$ does not belong to $\supp(\mu_n)$. Let $(r_n)$ be a sequence decreasing to zero such that 
$2r_n<\textrm{dist}(y_n,\supp(\mu_n))$. We set
$$\veps_n=r_n^n\textrm{ and }\nu_n=(1-\veps_n)\mu_n+\veps_n m_n.$$
Observe that $(\nu_n)$ remains dense in $\pk$. Moreover, the choice of $r_n$ ensures that
$$I_{\nu_n}(r_n,q)\geq\nu_n\big(\{y_n\})\nu_n\big(B(y_n,r_n)\big)^{q-1}=\nu_n\big(\{y_n\})^q=r_n^{-nq}.$$
Corollary \ref{CORTOPO3} gives us a real number $\delta_n>0$ such that any $\nu\in\pk$ satisfying $L(\nu,\nu_n)<\delta_n$
also verifies
\begin{eqnarray}
I_\nu(r_n/2,q)\geq C_q r_n^{-nq}.\label{EQUPPERLQINF0}
\end{eqnarray}
We then consider the dense $G_\delta$-set 
$$\mathcal R=\bigcap_{m\geq 0}\bigcup_{n\geq m}B_L(\nu_n,\delta_n).$$
Picking any $\nu\in\mathcal R$, we can find $n$ as large as we want and the corresponding $r_n$ as small as we want such that
(\ref{EQUPPERLQINF0}) holds true. Dividing by $(q-1)\log r_n$, which is positive, and taking the limsup, this yields $\dsupnuq=+\infty$, and this holds for a typical $\nu\in\pk$.

\subsection{Typical lower $L^q$-dimensions, $q\in(0,1)$}
We begin by applying the anti-Frostman lemma \ref{LEMANTIFROSTMAN} to get a measure $m\in\pk$ and some $C,r_0>0$ such that $m\big(B(x,r)\big)\geq Cr^t$ 
for all $x\in K$ and all $0<r\leq r_0$, where $t>\dboxsup(K)$. Let $(\mu_n)$ be a sequence of $\mathcal F(K)$ which is dense in $\pk$.
We write $\mu_n=\sum_{i=1}^N p_i\delta_{x_i}$ and we also set
$$r_n=\min\left(\frac 1n,\min_i(p_i)^{n(1-q)}\right),\ \veps_n=r_n^{t(1-q)/q}\textrm{ and }\nu_n=(1-\veps_n)\mu_n+\veps_n m.$$
Observe that $(\veps_n)$ goes to zero, so that $(\nu_n)$ remains dense in $\pk$. Let also $Y_n=\bigcup_{i=1}^N B(x_i,r_n)$. It is easy to check that
\begin{itemize}
 \item $\forall x\in Y_n,\ \nu_n\big(B(x,r_n)\big)\geq\min_i(p_i)/2$;
\item $\forall x\notin Y_n,\ \nu_n\big(B(x,r_n)\big)\geq C\veps_n r_n^t$.
\end{itemize}
Thus, one obtains
\begin{eqnarray*}
 I_{\nu_n}(r_n,q)& =&\int_{Y_n}\frac{d\nu_n(x)}{\nu_n\big(B(x,r_n)\big)^{1-q}}+\int_{Y_n^c}\frac{d\nu_n(x)}{\nu_n\big(B(x,r_n)\big)^{1-q}}\\
&\leq&C\left(\min(p_i)^{q-1}+\veps_n^qr_n^{-t(1-q)}\right).
\end{eqnarray*}
(in this proof, the letter $C$ denotes a positive real number which may change from line to line, remaining
always independent of $n$). The values that we impose on $r_n$ and $\veps_n$ ensure that
$$I_{\nu_n}(r_n,q)\leq C\left(\frac{1}{r_n^{1/n}}+1\right).$$
Now, Corollary \ref{CORTOPO3} gives us some $\delta_n>0$ such that any $\nu\in\mathcal P(K)$ with $L(\nu,\nu_n)<\delta_n$ satisfies
$$I_{\nu_n}(2r_n,q)\leq C\left(\frac{1}{r_n^{1/n}}+1\right).$$
It is now routine to prove that
$$\mathcal R=\bigcap_{m\geq 0}\bigcup_{n\geq m}B_L(\nu_n,\delta_n)$$
is the dense $G_\delta$-set we are looking for to prove that, typically, $\dinfmuq=0$ when $q\in(0,1)$.

\section{Typical $L^1$-dimension}

We now turn to the $L^1$-dimensions. Lemma \ref{LEMJENSEN} yields that, for a typical measure $\mu\in\pk$,
$$0\leq \dinfmu(1)\leq \dinfmu(1/2)=0\textrm{ and }s_u\leq \dsupmu(2)\leq \dsupmu(1)\leq\dsupmu(1/2)=s.$$
We now prove the nontrivial part of Theorem \ref{THMMAINL1}. We shall need several times the following analog of Lemma \ref{LEMJENSENLQ}.
\begin{lemma}\label{LEMJENSENL1}
Let $K$ be a compact subset of $\mathbb R^d$, let $A\subset K$, let $\mu\in\pk$ and let $r>0$. Then
$$\int_A\log\mu\big(B(x,2r)\big)d\mu(x)\geq-\mu\big(A(r)\big)\log\mathbf N_r(A)-e^{-1}.$$
\end{lemma}
\begin{proof}
We follow \cite{Ol07}. For brevity, write $N=\mathbf N_r(A)$. Let $x_1,\dots,x_N\in K$ be such that 
$A\subset \bigcup_{i=1}^N B(x_i,r)$. Put $E_1=B(x_1,r)$ and $E_i=B(x_i,r)\backslash \bigcup_{j=1}^{i-1}B(x_j,r)$ 
for $i=2,\dots,N$. Since for any $x\in E_i$, $E_i\subset B(x,2r)$, we may write
\begin{eqnarray*}
\int_A \log\mu\big(B(x,2r)\big)d\mu(x)&=&\sum_i \int_{A\cap E_i}\log \mu\big(B(x,2r)\big)d\mu(x)\\
&\geq&\sum_i \int_{A\cap E_i}\log \big(\mu(E_i\cap K)\big)d\mu(x)\\
&\geq&\sum_i \mu(E_i\cap K)\log\big(\mu(E_i\cap K)\big).
\end{eqnarray*}
Now, we apply Jensen's inequality to the convex function $\phi:(0,+\infty)\to\mathbb R$, $t\mapsto t\log t$. Observing that 
$\bigcup_{i=1}^N E_i\cap K\subset A(r)$ and that $\phi$ is bounded from below by $e^{-1}$, we get
\begin{eqnarray*}
\int_A \log\mu\big(B(x,2r)\big)d\mu(x)&\geq&N\sum_{i=1}^N \frac1N\phi\big(\mu(E_i\cap K)\big)\\
&\geq&N \phi\left(\sum_i \frac 1N \mu(E_i\cap K)\right)\\
&\geq&\phi\left(\mu\left(\bigcup_i E_i\cap K\right)\right)-\mu\left(\bigcup_i E_i\cap K\right)\log N\\
&\geq&-e^{-1}-\mu\big(A(r)\big)\log N.
\end{eqnarray*}
\end{proof}
We divide the proof of Theorem \ref{THMMAINL1} into four parts. For convenience, we set :
$$\sconv=\dboxsupconv(K)\textrm{ and }\sconvmax=\dboxsupconvmax(K).$$
\subsection{The lower bound}
Let $t<\sconv$. One intends to show that $\dsupmu(1)\geq t$ for a typical measure $\mu\in\pk$. 
Let $(x_i)_{i\geq 1}$ be a dense sequence of distinct points in $K$ and let $n\geq 1$. Let $\rho_n>0$ be such that 
$4\rho_n<\min(\|x_i-x_j\|;\ 1\leq i<j\leq n)$ and $\rho_n<1/n$. Let finally $p_1,\dots,p_n>0$ be such that $\sum_i p_i=1$. 
By definition of $\sconv$, one may find $r_n<\rho_n$ such that, setting $P_i=\mathbf P_{r_n}\big(K\cap B(x_i,\rho_n)\big)$,
$$\sum_{i=1}^n p_i\log(P_i)\geq -t\log r_n.$$
We then consider, for a fixed $i$ in $1,\dots,n$, points $x_i^j$, $j=1,\dots,P_i$, which are the centers
in $K\cap B(x_i,\rho_n)$ of disjoint balls of radius $r_n$. We observe that our choice of $\rho_n$ ensures that 
all balls $B(x_i^j,\rho_n)$, $i=1,\dots,n$, $j=1,\dots,P_i$ are pairwise disjoint. We set
$$\mu_{n,(p_i)}=\sum_{i=1}^n \frac{p_i}{P_i}\sum_{j=1}^{P_i}\delta_{x_i^j}$$
and we majorize $I_{\mu_{n,(p_i)}}(r_n,1)$:
\begin{eqnarray*}
 I_{\mu_{n,(p_i)}}(r_n,1)&=&\sum_{i=1}^n \frac{p_i}{P_i}\sum_{j=1}^{P_i}\log\left(\mu_{n,(p_i)}\big(B(x_i^j,r_n)\big)\right)\\
&=&\sum_{i=1}^n p_i\log(p_i)-\sum_{i=1}^n p_i\log(P_i)\\
&\leq&t\log r_n.
\end{eqnarray*}
Now, Corollary \ref{CORTOPOL1} gives us a real number $\delta_{n,(p_i)}>0$ such that $L(\nu,\mu_{n,(p_i)})<\delta_{n,(p_i)}$
implies $I_\nu(r_n/2,1)\leq \log 2+\left(1-\frac 1n\right)t\log r_n$.
We conclude by setting
\begin{eqnarray*}
 \Lambda_n&=&\left\{\mu_{n,(p_i)};\ p_i>0,\ \sum_i p_i=1\right\}\\
\mathcal R_n&=&\bigcup_{\mu_{n,(p_i)}\in \Lambda_n}B_L(\mu_{n,(p_i)},\delta_{n,(p_i)})\\
\mathcal R&=&\bigcap_{m\geq 1}\bigcup_{n\geq m}\mathcal R_n.
\end{eqnarray*}
$\mathcal R$ is a dense $G_\delta$ subset of $\pk$ and any $\nu\in\mathcal R$ satisfies
$$\overline{D}_\nu(1)\geq\limsup_{n\to+\infty}\left(1-\frac1n\right)t=t.$$

\subsection{The upper bound}
Let $t>\sconvmax$. One intends to show that there exists a dense and open set of measures $\mu$ of $\pk$
such that $\dsupmu(1)\leq t$. We first fix $t_0\in(\sconvmax,t)$ and $\eta>0$ such that $(1+2\eta)t_0<t$.
Let $(y_n)$ be a dense sequence of distinct points in $K$ and let $(\rho_n)$ be a sequence of $(0,1)$ such that
$$\left\{
\begin{array}{l}
 \rho_n\leq \frac1n\\
B(y_i,2\rho_n)\cap B(y_j,2\rho_n)=\varnothing\textrm{ for any }i\neq j\textrm{ in }\{1,\dots,n\}.
\end{array}\right.$$
Let $N\geq 1$ and let $n\leq N$. By definition of $\sconvmax$, we know that we may find an integer
$M_{n,N}>0$, elements $x_n^1,\dots,x_n^{M_{n,N}}$ contained in $B(y_n,\rho_n)$, a real number $\delta_{n,N}\in(0,\rho_N)$
and real numbers $p_n^1,\dots,p_n^{M_{n,N}}$ in $(0,1)$ with $\sum_i p_n^i=1$ satisfying
$$\sum_i \frac{p_n^i \log \mathbf N_r\big(K\cap B(x_n^i,\delta_{n,N})\big)}{-\log r}\leq t_0$$
provided $r$ is small enough. Reducing $\delta_{n,N}$ if necessary, we can always assume that
the balls $B(x_n^i,2\delta_{n,N})$, for $i=1,\dots,M_{n,N}$, are pairwise disjoint and that they are contained
in $B(y_n,\rho_N)$. Let also $\mathcal Q_N$ be defined by
$$\mathcal Q_N=\left\{(q_1,\dots,q_N)\in(0,1);\ \sum_i q_i=1\right\}.$$
For any $\mathbf q=(q_1,\dots,q_N)$ in $\mathcal Q_N$, we also fix $\veps_{N,\mathbf q}\in(0,1)$ satisfying
$$\left\{
\begin{array}{rcl}
 (d+1)\veps_{N,\mathbf q}&\leq& \eta t_0\\
\veps_{N,\mathbf q}(1-p_n^i q_n)&\leq&\eta p_n^i q_n\textrm{ for any }n=1,\dots,N\textrm{ and any }i=1,\dots,M_{n,N}.
\end{array}\right.$$
We finally set, for $N\geq 1$ and $\mathbf q\in \mathcal Q_N$, 
\begin{eqnarray*}
 \mathcal U_{N,\mathbf q}=\Bigg\{\mu\in\pk&;\ &\forall n=1,\dots,N,\ \forall i=1,\dots,M_{n,N}, \\
&&\mu\big(B(x_n^i,\delta_{n,N}/2)\big)>(1-\veps_{N,\mathbf q})p_n^i q_n\textrm{ and }\\
&&\mu\big(B(y_n,\rho_N) \big)>(1-\veps_{N,\mathbf q})q_n\Bigg\}.
\end{eqnarray*}
Each $\mathcal U_{N,\mathbf q}$ is open by Lemma \ref{LEMTOPOOPEN} and $\bigcup_{N,\mathbf q}\mathcal U_{N,\mathbf q}$ is dense
by Lemma \ref{LEMDENS1}. Let us show that any $\mu\in\mathcal U_{N,\mathbf q}$ satisfies $\dsupmu(1)\leq t$. 
We set 
$$Y_N=K\backslash\bigcup_{n=1}^N \bigcup_{i=1}^{M_{n,N}}B(x_n^i,\delta_{n,N})$$
and we apply Lemma \ref{LEMJENSENL1} to each $B(x_n^i,\delta_{n,N})$ and to $Y_N$. Taking into account
that we have a partition of $K$, we thus obtain
\begin{eqnarray*}
 I_{\mu}(2r,1)\!&\!\geq\!&-\sum_{n=1}^N \sum_{i=1}^{M_{n,N}}\mu\big(B(x_n^i,\delta_{n,N}+r)\big)\log\mathbf N_r\big(K\cap B(x_n^i,\delta_{n,N})\big)\\
&&-\mu\big(Y_N(r)\big)\log\mathbf N_r(Y_N)+C_\mu,
\end{eqnarray*}
where $C_\mu$ does not depend on $r$. 
We first look at the last term. Provided $r$ is small enough, 
$$Y_N(r)\subset K\backslash \bigcup_{n=1}^N \bigcup_{i=1}^{M_{n,N}}B(x_n^i,\delta_{n,N}/2)$$
so that
\begin{eqnarray*}
 \mu\big(Y_N(r)\big)&\leq&1-\sum_{n}\sum_{i}(1-\veps_{N,\mathbf q})p_n^iq_n\\
&\leq&\veps_{N,\mathbf q}.
\end{eqnarray*}
On the other hand, since $Y_N$ is contained in $\mathbb R^d$, which has dimension $d$, we know that
$$\log\big(\mathbf N_r(Y_N)\big)\leq -(d+1)\log r$$
provided $r$ is small enough. Thus, taking into account the value of $\veps_{N,\mathbf q}$, 
$$-\mu\big(Y_N(r)\big)\log \mathbf N_r(Y_N)\geq\eta t_0\log r.$$

\smallskip

The analysis of the first term is slightly more delicate. We first observe that, provided $r$ is small enough,
\begin{eqnarray*}
 \mu\big(B(x_n^i,\delta_{n,N}+r)\big)&\leq&\mu\big(B(x_n^i,2\delta_{n,N})\big)\\
&\leq&\mu\big(B(y_n,\rho_N)\big)-\sum_{j\neq i}\mu\big(B(x_n^j,\delta_{n,N})\big)\\
&\leq&1-\sum_{m\neq n}\mu\big(B(y_m,\rho_N)\big)-\sum_{j\neq i}(1-\veps_{N,\mathbf q})p_n^jq_n\\
&\leq&1-(1-\veps_{N,\mathbf q})\sum_{m\neq n}q_m-(1-\veps_{N,q})(1-p_n^i)q_n\\
&\leq&1-(1-\veps_{N,\mathbf q})(1-q_n)-(1-\veps_{N,q})(1-p_n^i)q_n\\
&\leq&p_n^iq_n+\veps_{N,\mathbf q}(1-p_n^iq_n)\\
&\leq&(1+\eta)p_n^iq_n.
\end{eqnarray*}
Thus,
\begin{eqnarray*}
 I_\mu(2r,1)&\geq& -(1+\eta)\sum_nq_n\sum_i p_n^i\log\mathbf N_r\big(K\cap B(x_n^i,\delta_{n,N})\big)+\eta t_0\log r+C_\mu\\
&\geq&(1+2\eta)t_0\log r+C_\mu,
\end{eqnarray*}
provided again that $r$ is small enough. Dividing by $\log r$ and taking the limsup, we find 
$$\dsupmu(1)\leq (1+2\eta)t_0\leq t.$$

\subsection{Optimality of the lower bound}
Suppose that there exists $t>\sconv$ such that 
$\mathcal R=\big\{\mu\in\pk;\ \dsupmu(1)\geq t\big\}$ is residual.
We shall arrive at a contradiction. The argument follows rather closely that of the previous
subsection.
We fix $t_0\in(\sconv,t)$ and $\eta>0$ such that $(1+2\eta)t_0<t$. By definition of $\sconv$,
one may find $x_1,\dots,x_N\in K$, $\rho>0$ with $4\rho<\min(\|x_i-x_j\|;\ i\neq j)$
and $p_1,\dots,p_N>0$ with $\sum_i p_i=1$ such that any $r>0$ sufficiently small verifies
$$\sum_{i=1}^N p_i\log\Big(\mathbf N_r\big(K\cap B(x_i,\rho)\big)\Big)\leq (-\log r)t_0.$$
Now, let $\veps>0$ satisfying the following conditions:
$$\left\{
\begin{array}{rcl}
 (d+1)\veps&\leq& \eta t_0\\
\veps&\leq&\eta\frac{p_i}{1-p_i}\textrm{ for any }i=1,\dots,N.
\end{array}\right.$$
Let finally 
$$\mathcal U=\bigcap_{i=1}^N \left\{\mu\in\pk;\ \mu\big(B(x_i,\rho/2)\big)>(1-\veps)p_i\right\}.$$
$\mathcal U$ is open and nonempty and one can pick $\mu\in\mathcal U\cap \mathcal R$. We set 
$Y=K\backslash \bigcup_{i=1}^N B(x_i,\rho)$ and we apply Lemma \ref{LEMJENSENL1} to get
\begin{eqnarray*}
 I_\mu(2r,1)&\geq&-\sum_{i=1}^N \mu\big(B(x_i,\rho+r)\big)\log \mathbf N_r\big(K\cap B(x_i,\rho)\big)-\mu\big(Y(r)\big)\log\mathbf N_r(Y)+C_\mu\\
&\geq&-\sum_{i=1}^N \mu\big(B(x_i,2\rho)\big)\log \mathbf N_r\big(K\cap B(x_i,\rho)\big)-\mu\big(Y(\rho/2)\big)\log\mathbf N_r(Y)+C_\mu
\end{eqnarray*}
provided $r<\rho/2$. Now we observe that
\begin{eqnarray*}
 \mu\big(B(x_i,2\rho)\big)&\leq&1-\sum_{j\neq i}\mu\big(B(x_j,\rho)\big)\\
&\leq&1-\sum_{j\neq i}p_j(1-\veps)\\
&\leq&p_i+\veps(1-p_i)\leq(1+\eta)p_i.
\end{eqnarray*}
In the same vein, we also get
$$\mu\big(Y(\rho/2)\big)\leq 1-\sum_j \mu\big(B(x_j,\rho/2)\big)\leq\veps,$$
whereas, provided $r$ is small enough,
$$\log \mathbf N_r(Y)\leq -(d+1)\log r.$$
Coming back to $I_\mu(2r,1)$ and summarizing the previous results, we conclude that
\begin{eqnarray*}
 I_\mu(2r,1)&\geq&-\sum_{i=1}^N (1+\eta)p_i\log\mathbf N_r\big(K\cap B(x_i,\rho)\big)+(d+1)\veps\log r+C_\mu\\
&\geq&(1+2\eta)t_0\log r +C_\mu
\end{eqnarray*}
where $C_\mu$ does not depend on $r$. Dividing by $\log r$ and taking the limsup, 
we obtain that $\dsupmu(1)\leq (1+2\eta)t_0<t$, a contradiction with $\mu\in\mathcal R$.

\subsection{Optimality of the upper bound}
Suppose that there exists $t<\sconvmax$ such that $\mathcal R=\big\{\mu\in\pk;\ \dsupmu(1)\leq t\big\}$ is residual.
Let $t_0\in(t,\sconvmax)$ and let $y\in K$, $\rho>0$ be such that, setting $E=K\cap B(y,\rho)$, 
$\dboxsupconv(E)>t_0$. Let also $\beta>0$ be such that $\frac{(1-\beta)^3}{1+\beta}t_0>t$. 

We start from a sequence $(\mu_n)\subset\fk$ which is dense in $\pk$ and such that $\mu_n(E)>0$ for any $n\geq 1$. We fix $n\geq 1$ and we write
$$\mu_n=\sum_{i=1}^q q_i\delta_{y_i}+\sum_{i=1}^p p_i\delta_{x_i}$$
where $x_i\in E$, $y_i\notin E$, $p_i,q_i>0$. Let $\alpha>0$ be such that $4\alpha<\min(\|x_i-x_j\|;\ i\neq j)$, 
$4\alpha<\min(\|y_i-y_j\|;\ i\neq j)$, $4\alpha<\min(\|x_i-y_j\|)$ and $\alpha<1/n$.
By definition of the convex upper box dimension, we know that one can find $r_n<\min(\alpha,1/n)$ as small as we want such that
$$\sum_{i=1}^p p_i \log \mathbf P_{r_n}\big(E\cap B(x_i,\alpha)\big) \geq (-\log r_n)t_0\sum_{i=1}^p p_i=(-\log r_n)t_0\mu_n(E).$$
For simplicity, we set $P_i=\mathbf P_{r_n}\big(E\cap B(x_i,\alpha)\big)$ and we consider
points $x_{i}^1,\dots,x_{i}^{P_i}$ all in $E\cap B(x_i,\alpha)$ such that the balls
$B(x_{i}^j,r_n)$ are disjoint. We then set
\begin{eqnarray*}
 \nu_{n,i}&=&\frac{1}{P_i}\sum_{j=1}^{P_i}\delta_{x_i^j}\\
\nu_n&=&\sum_{i=1}^q q_i\delta_{y_i}+\sum_{i=1}^p p_i\nu_{n,i}.
\end{eqnarray*}
$\nu_n$ is close to $\mu_n$. Indeed, for any $f\in\textrm{Lip}(K)$, 
$$\left|\int_K fd\nu_n-\int_K fd\mu_n\right|\leq\sum_{i=1}^p p_i\|x_i-x_i^j\|\leq\alpha\leq\frac 1n.$$
In particular, the sequence $(\nu_n)$ keeps dense in $\pk$. Furthermore, $\nu_n(E)=\mu_n(E)>0$ and
\begin{eqnarray*}
 I_{\nu_n}(r_n,1)&=&\sum_{i=1}^q q_i\log q_i+\sum_{i=1}^p \frac{p_i}{P_i}\sum_{j=1}^{P_i}\log\left(\frac{p_i}{P_i}\right)\\
&=&\sum_{i=1}^q q_i\log q_i+\sum_{i=1}^p p_i\log(p_i)-\sum_{i=1}^p p_i\log(P_i)\\
&\leq&\sum_{i=1}^q q_i\log(q_i)+\sum_{i=1}^p p_i\log(p_i)+\log(r_n)t_0\nu_n(E).
\end{eqnarray*}
Provided $r_n$ is small enough, this implies
$$I_{\nu_n}(r_n,1)\leq (1-\beta)t_0\log(r_n)\nu_n(E).$$
Now, Lemma \ref{LEMTOPO1} and Corollary \ref{CORTOPOL1} give us a real number $\delta_n>0$ such that $L(\nu,\nu_n)<\delta_n$
implies
$$\left\{\begin{array}{rcl}
  I_{\nu}(r_n/2,1)&\leq&(1-\beta) I_{\nu}(r_n,1)+\log 2\\\
\nu(E)&\leq& \nu_n\big(E(\alpha)\big)+\beta\nu_n(E)=(1+\beta)\nu_n(E).        
         \end{array}\right.
$$
Thus,
$$I_{\nu}(r_n/2,1)\leq\frac{(1-\beta)^2}{1+\beta}t_0\log (r_n)\nu(E)+\log 2.$$
We define 
$$\mathcal R'=\bigcap_{m\geq 1}\bigcup_{n\geq m}B_L(\nu_n,\delta_n)$$
which is a dense $G_\delta$-subset of $\pk$ such that any $\nu\in \mathcal R'$
verifies $\dsupnu(1)\geq \frac{(1-\beta)^2}{1+\beta}t_0\nu(E)$.
Finally, we considered
$$\mathcal U=\big\{\mu\in\pk;\ \mu(E)>1-\beta\big\}$$
which is open and nonempty. Then any measure $\mu$ in the nonempty intersection $\mathcal R\cap\mathcal R'\cap\mathcal U$ satisfies
the contradictory conditions 
$$\dsupmu(1)\leq t\textrm{ and }\dsupmu(1)\geq\frac{(1-\beta)^3}{1+\beta}t_0>t.$$

\subsection{Examples}

To apply Theorem \ref{THMMAINL1}, we need to be able to compute the dimensions which are involved. 
This is indeed possible for simple compact spaces. 

\begin{example}
 Let $K=\{0\}\cup[1,2]$. Then a typical measure $\mu\in\pk$ satisfies $\dsupmu(1)\in[0,1]$ and this interval is the best possible.
\end{example}
\begin{proof}
 That $\sconv=0$ is easy. It suffices to take in the definition of the convex upper box dimension $N=1$ and $x_1=0$. 
To prove that $\sconvmax=1$, one may observe that $\dboxsupconv([1,2])=1$.
\end{proof}
\begin{example}
 Let $E$ and $F$ be the two Cantor sets defined in Section \ref{SECCANTOR} and let $K=E\cup F$. Then a typical
measure $\mu\in\pk$ satisfies $\dsupmu(1)\in\left[\frac{13\log 3}{6\log 100},\frac{5\log 3}{2\log 100}\right]$ and this interval
is the best possible.
\end{example}
An interesting feature of the previous example is that $\sconv>s_u$.
\begin{proof}
 By homogeneity of the two Cantor sets, to compute $\sconv$, we just need to compute
$$\inf_{\alpha+\beta=1}\limsup_{n\to+\infty}\frac{\alpha \mathbf N_{100^{-n}}(E)+\beta\mathbf N_{100^{-n}}(F)}{n\log 100}.$$
From the construction of the Cantor sets, it is easy to check that, for a fixed choice of $\alpha$ and $\beta$, the 
limsup will be attained along either the sequence $(2n_k)_k$ or the sequence $(4n_k)_k$. Moreover
$$\begin{array}{ll}
 \mathbf N_{100^{-2n_k}}(E)=9^{n_k}27^{n_k}=3^{5n_k}&\mathbf N_{100^{-2n_k}}(F)=3^{4n_k}\\
 \mathbf N_{100^{-4n_k}}(E)=3^{8n_k}&\mathbf N_{100^{-4n_k}}(F)=3^{9n_k}.
\end{array}$$
Thus,
\begin{eqnarray*}
 \sconv&=&\inf_{\alpha+\beta=1}\max\left((5\alpha+4\beta)\frac{\log 3}{2\log 100},(8\alpha+9\beta)\frac{\log 3}{4\log 100}\right)\\
&=&\frac{\log 3}{4\log 100}\inf_{\alpha\in[0,1]}(8+2\alpha,9-\alpha).
\end{eqnarray*}
The minimum is obtained for $\alpha=1/3$ (when the two terms are equal) so that
$$\sconv=\frac{13\log 3}{6\log 100}.$$
That $\sconvmax=\dimp(K)=\frac{5\log 3}{2\log 100}$ is easier. Indeed, we can restrict ourselves to $E$ where, by homogeneity,
$\dboxsupconv(E)=\dboxsup(E)$.
\end{proof}
\begin{example}
 Let $K$ be a Ahlfors-regular compact set with Hausdorff dimension $s$. Then a typical measure $\mu\in\pk$ satisfies
$\dsupmu(1)=s$.
\end{example}
\begin{proof}
 For such a compact set,
$$s\leq\inf_{x\in K}\dboxinfloc(x,K)\leq s_{\rm conv}\leq s_{\rm conv}^{\rm max}\leq\dboxsup(K)=s.$$
\end{proof}

\section{Concluding remarks}

\subsection{Case $q=\pm\infty$}
The $L^q$-dimensions of a probability measure $\mu$ have also a meaning for $q=-\infty$ and $q=+\infty$.
Their definitions are
\begin{eqnarray*}
  \dsupmu(-\infty)&=&\limsup_{r\to 0}\frac{\log\inf_{x\in\supp(\mu)}\mu\big(B(x,r)\big)}{\log r}\\
 \dinfmu(-\infty)&=&\liminf_{r\to 0}\frac{\log\inf_{x\in\supp(\mu)}\mu\big(B(x,r)\big)}{\log r}\\
  \dsupmu(+\infty)&=&\limsup_{r\to 0}\frac{\log\sup_{x\in\supp(\mu)}\mu\big(B(x,r)\big)}{\log r}\\
 \dinfmu(+\infty)&=&\liminf_{r\to 0}\frac{\log\sup_{x\in\supp(\mu)}\mu\big(B(x,r)\big)}{\log r}.
  \end{eqnarray*}
In \cite{Ol07}, Olsen shows that a typical $\mu\in\pk$ satisfies
$$\dinfmu(+\infty)=0\textrm{ and }\inf_{x\in K}\dboxinfloc(x,K)\leq\dsupmu(+\infty)\leq\inf_{x\in K}\dboxsuploc(x,K).$$
Our methods allow us to determine the exact typical value of the $L^q$-dimensions, for $q=-\infty$ and $q=+\infty$.
\begin{theorem}
 Let $K$ be an infinite compact subset of $\mathbb R^d$. Then a typical $\mu\in\pk$ satisfies
$$\begin{array}{ll}
   \dinfmu(+\infty)=0\quad&\dsupmu(+\infty)=\dboxsuplocunif(K)\\
\dinfmu(-\infty)=\dboxinf(K)\quad&\dsupmu(-\infty)=+\infty.
  \end{array}$$
\end{theorem}
Among these values, that of $\dinfmu(-\infty)$ is surprising, because it is different from the value of $\dinfmuq$ when $q<0$.
\begin{proof}
 For commodity reasons, we set $s_u=\dboxsuplocunif(K)$. It is easy to show that, for any $q\in\mathbb R$,
$$\dinfmu(+\infty)\leq \dinfmuq\leq \dinfmu(-\infty)\textrm{ and }\dsupmu(+\infty)\leq \dsupmuq\leq \dsupmu(-\infty).$$
Thus, Theorem \ref{THMMAINLQSPECTRUM} already implies that $\dsupmu(+\infty)\leq s_u$ and that $\dsupmu(-\infty)=+\infty$
for a typical $\mu\in\pk$. Thus, it remains to show the typical (in)equalities $\dsupmu(+\infty)\geq s_u$
and $\dinfmu(-\infty)=\dboxinf(K)$. We begin by showing that $\dsupmu(+\infty)\geq s_u$
for a typical $\mu\in\pk$. Let $t<s_u$ and let us apply the results of Section \ref{SECUPPERUNIFORMBOXANDLQDIM}.
They provide a dense $G_\delta$-set $\mathcal R$ and a sequence $(r_n)$ going to zero such that, for any $\nu\in\mathcal R$,
we can find $n$ as large as we want such that 
$$\nu\big(B(x,r_n)\big)\leq Cr_n^t.$$
This immediately yields $\dsupnu(+\infty)\geq t$.

\smallskip

We now prove that a typical $\mu\in\pk$ satisfies $\dinfmu(-\infty)\geq\dboxinf(K)$. For $n\geq 1$, let $P=\mathbf P_{2^{-n}}(K)$
be the maximal number of balls of radius $2^{-n}$ with center in $K$ which do not intersect. Let $B(x_1,2^{-n}),\dots,B(x_P,2^{-n})$
be such a family of balls. We set $U_{j,n}=B(x_j,2^{-n})$, $U'_{j,n}=B(x_j,2^{-(n+1)})$ and
\begin{eqnarray*}
 \mathcal R_n&=&\left\{\mu\in\mathcal P(K);\ \forall j\in\left\{1,\dots,\mathbf P_{2^{-n}}(K)\right\},\ \mu(U'_{j,n})>0\right\}\\
\mathcal R&=&\bigcap_{n\geq 1}\mathcal R_n.
\end{eqnarray*}
Each $\mathcal R_n$ is dense by Lemma \ref{LEMDENS2} and open by Lemma \ref{LEMTOPOOPEN}, so that $\mathcal R$
is a dense $G_\delta$-subset of $\pk$. Pick now any $\mu\in\mathcal R$ and let $r\in(0,1)$. There exists $n\geq 1$
such that $2^{-(n+1)}\leq r<2^{-n}$. Moreover,
$$\sum_{j=1}^{\mathbf P_{2^{-n}}(K)}\mu(U_{j,n})\leq 1$$
so that there exists $j\in\{1,\dots,\mathbf P_{2^{-n}}(K)\}$ with $\mu(U_{j,n})\leq \big(\mathbf P_{2^{-n}}(K)\big)^{-1}$.
Now, since $\mu(U'_{j,n})>0$, one may find $y_j\in\supp(\mu)\cap U'_{j,n}$. Noticing
that $B(y_j,r/2)$ is contained in $U_{j,n}$, we get
$$\log\mu\big(B(y_j,r/2)\big)\leq \log \mu(U_{j,n})\leq -\log \mathbf P_{2^{-n}}(K)$$
which itself yields
$$\frac{\log\inf_{x\in\supp(\mu)} \mu\big(B(x,r/2)\big)}{\log(r/2)}\geq \frac{-\log \mathbf P_{2^{-n}}(K)}{\log(r/2)}\geq\frac{\log \mathbf P_{2^{-n}(K)}}{(n+2)\log 2}.$$
Taking the liminf, this shows that $\dinfmu(-\infty)\geq \dboxinf(K)$.

\smallskip

We finally prove that $\dinfmu(-\infty)\leq \dboxinf(K)$. Let $t>t'>\dboxinf(K)$. Lemma \ref{LEMANTIFROSTMANLOWERBOX} gives us, for 
each $n\geq 1$, a measure $m_n\in\pk$  and a real number $r_n\in(0,1/n)$ such that $m_n\big(B(x,r_n/2)\big)\geq Cr_n^{t'}$
for each $x\in K$, where $C$ just depends on $t'$. Let now $\mu\in\fk$ and let us set
$$\nu_{\mu,n}=\big(1-r_n^{(t-t')}\big)\mu+r_n^{t-t'}m_n.$$
For any $x\in K$, $\nu_{\mu,n}\big(B(x,r_n)\big)\geq Cr_n^t$. By Lemma \ref{LEMTOPO1}, we may find $\delta_n>0$ such that any $\nu\in\pk$ 
with $L(\nu,\nu_{\mu,n})<\delta_n$ satisfies
$$\nu\big(B(x,2r_n)\big)\geq 2Cr_n^t.$$
 We then consider 
$$\mathcal R=\bigcap_{n\geq 1}\bigcup_{\mu\in\fk}B(\nu_{\mu,n},\delta_n)$$
which is a dense $G_\delta$-set. For any $\nu\in\mathcal R$, there exists $r_n$ as small as desired such that
$$\frac{\log\inf_{x\in\supp\mu}\nu\big(B(x,2r_n)\big)}{\log r_n}\leq t+\frac{\log(2C)}{\log r_n},$$
so that $\dinfnu(-\infty)\leq t$.
\end{proof}

\subsection{On the optimality}

It is natural to ask whether the inequalities appearing in the second column of Theorem \ref{THMMAINLQSPECTRUM} are optimal or not.
The third column gives the answer, except for $\dsupmuq$, $q\geq 1$ and $\dinfmuq$, $q\leq 0$. It is very easy to exhibit a measure satisfying
$\dinfmuq=0$, $q<0$: any Dirac mass does the job. The situation is different for $\dsupmuq$, $q\geq 1$; we just know that
$\dsupmu(1)\leq\dimp(K)$, but we do not know whether this is always the best bound.

\begin{question}
 What is the biggest possible valut of $\dsupmuq$, $q\geq 1$?
\end{question}

It is conceivable that, at least for $q=1$, we need a kind of convex version of the packing dimension.

\subsection{$L^q$-dimension along subsequences}
When, for a measure $\mu\in\pk$, $\dinfmuq<\dsupmuq$, it is natural to ask the following question: for which $\tau\in(\dinfmuq,\dsupmuq)$
can we find a sequence $(r_n)$ going to zero such that 
$$\frac{\log I_\mu(r_n,q)}{(q-1)\log r_n}\to\tau ?$$
The most interesting case is when all values of $(\dinfmuq,\dsupmuq)$ can be attained. 
For continuity reasons, this is always the case.
\begin{proposition}\label{PROPLQSUBSEQUENCES}
 Let $K$ be a compact subset of $\mathbb R^d$, let $\mu\in\pk$ and let $q\in\mathbb R$. 
Then, for any $\tau\in(\dinfmuq,\dsupmuq)$, there exists a sequence $(r_n)$ going to zero such that
$$\frac{\log I_\mu(r_n,q)}{(q-1)\log r_n}\to\tau\textrm{ when $q\neq 1$},\ \frac{\log I_\mu(r_n,1)}{\log r_n}\to\tau\textrm{ otherwise.}$$
\end{proposition}
The proof of this proposition is based on an application of the intermediate value theorem. Unfortunately, the map $r\mapsto \frac1{q-1}\times\frac{\log I_\mu(r,q)}{\log r}$
does not need to be continuous. However, an enhancement of the intermediate value theorem to semicontinuous functions will be sufficient in our context.
It can be found e.g. in \cite{Gui95}.
\begin{lemma}\label{LEMIVT}
 Let $f:[u,v]\to\mathbb R$ be upper semicontinuous on the right and lower semicontinuous on the left.
Suppose moreover that $f(u)\geq f(v)$. Then for any $\lambda\in\big(f(v),f(u)\big)$, there
exists $x\in[u,v]$ such that $f(x)=\lambda$.
\end{lemma}
The regularity of $I_ \mu(\cdot,q)$ will depend on the position of $q$ with respect to 1.
\begin{lemma}\label{LEMREGIMURQ}
 Let $\mu\in\pk$ and let $q\in\mathbb R$. Suppose moreover that, for any $r>0$, $I_\mu(r,q)\neq+\infty$.
\begin{enumerate}[(a)]
 \item For $q> 1$, $r\mapsto\log I_\mu(r,q)$ is negative, nondecreasing and continuous on the left;
 \item For $q=1$, $r\mapsto I_\mu(r,1)$ is negative, nondecreasing and continuous on the left;
 \item For $q<1$, $r\mapsto \log I_\mu(r,q)$ is positive, nonincreasing and continuous on the left.
\end{enumerate}
\end{lemma}
\begin{proof}
 We just prove the continuity statement. We fix $r>0$ and we pick a sequence $(r_n)$ increasing to 
$r$. Then $\big(\mu(B(x,r_n)\big)$ increases to $\mu\big(B(x,r)\big)$ for any $x\in\supp(\mu)$. 
Thus (a) follows from the monotone convergence theorem, whereas (b) and (c) follow from Lebesgue's theorem.
\end{proof}
In order to apply Lemma \ref{LEMIVT} to $r\mapsto \frac1{q-1}\times\frac{\log I_\mu(r,q)}{\log r}$, we need a last lemma.
\begin{lemma}\label{LEMPRODUCT}
 Let $I$ be an interval, let $g:I\to (-\infty,0)$ be continuous and let $f:I\to\mathbb (-\infty,0)$ be nondecreasing.
Then $fg$ is upper semicontinuous on the right.
\end{lemma}
\begin{proof}
 Let $r_0\in I$, let $\veps>0$ and let $\eta>0$ such that any $r\in [r_0,r_0+\eta)\cap I$ satisfies
$$g(r_0)\leq g(r)+\veps.$$
We also know that for these values of $r$,
$$f(r_0)\leq f(r)\leq 0,$$
so that
\begin{eqnarray*}
 f(r_0)g(r_0)&\geq&f(r)g(r_0)\\
&\geq&f(r)g(r)+\veps f(r)\\
&\geq&f(r)g(r)+\veps f(r_0).
\end{eqnarray*}
This shows that $\limsup_{r\to r_0^+}f(r)g(r)\leq f(r_0)g(r_0)$, which means that $fg$ is upper semicontinuous on the right.
\end{proof}

\begin{proof}[Proof of Proposition \ref{PROPLQSUBSEQUENCES}]
 Let $\tau\in (\dinfmuq,\dsupmuq)$ (if $\dinfmuq=\dsupmuq$), there is nothing to say. Define 
$$\phi(r)=\left\{\begin{array}{ll}\displaystyle
                  \frac1{q-1}\times\frac{\log I_\mu(r,q)}{\log r}&\textrm{ if }q\neq 1\\[0.35cm]
\displaystyle\frac{I_\mu(r,1)}{\log r}&\textrm{ if }q=1.
                 \end{array}\right.$$
By the above lemmas, $\phi$ is continuous on the left and upper semicontinuous on the right. Moreover, for any $\veps>0$,
one may find $0<r_1<r_2<\veps$ such that
$$\phi(r_1)>\tau\textrm{ and }\phi(r_2)<\tau.$$
By Lemma \ref{LEMIVT}, one can find $r_3\in(r_1,r_2)$ such that
$$\phi(r_3)=\tau.$$
The conclusion of Proposition \ref{PROPLQSUBSEQUENCES} follows easily.
\end{proof}
Combining this proposition with Theorem \ref{THMMAINLQSPECTRUM} yields:
\begin{corollary}
 Let $K$ be an infinite compact subset of $\mathbb R^d$. Write
\begin{eqnarray*}
 s_{\rm sep}&=&\bsi(K)\\
s_u&=&\dboxsuplocunif(K)\\
s&=&\dboxsup(K).
\end{eqnarray*}
Define, for $q\in\mathbb R\backslash\{1\}$, 
$$E_q=\left\{
\begin{array}{ll}
[0,s_u]&\textrm{ for }q> 1\\

[0,s]&\textrm{ for }q\in(0,1)\\

[s_{\rm sep},s]&\textrm{ for }q=0\\

[s_{\rm sep},+\infty)&\textrm{ for }q<0.

\end{array}\right.$$
Then a typical measure $\mu\in\mathcal P(K)$ satisfies, for any $q\in\mathbb R\backslash\{1\}$, for any $\tau\in E_q$, there exists a sequence $(r_n)$ going to zero such that
$$\frac{\log I_\mu(r_n,q)}{(q-1)\log r_n}\to \tau.$$
\end{corollary}

\subsection{Prevalence}

The notion of genericity used in this paper is in the sense of Baire theorem. One could
also consider other notions, in particular that of prevalence. In \cite{Ol10},
Olsen studied the typical values of $\dinfmuq$ and $\dsupmuq$ for $q\geq 0$ for this notion
of genericity. Surprizingly enough, the results are rather different from
the results in the Baire point of view.

\smallskip

{\bf Problem.} What can be said on $\dinfmuq$ and $\dsupmuq$ for a prevalent measure, when $q<0$?
\providecommand{\bysame}{\leavevmode\hbox to3em{\hrulefill}\thinspace}
\providecommand{\MR}{\relax\ifhmode\unskip\space\fi MR }
\providecommand{\MRhref}[2]{%
  \href{http://www.ams.org/mathscinet-getitem?mr=#1}{#2}
}
\providecommand{\href}[2]{#2}

\end{document}